\documentclass[a4paper, 11pt, reqno]{amsart}

\usepackage[utf8]{inputenc}
\usepackage[english]{babel}

\usepackage{amsfonts}
\usepackage{amssymb}
\usepackage{amsmath}
\usepackage{amsthm}
\usepackage{mathtools}
\usepackage{bbm}                
\usepackage{mathrsfs}           

\usepackage{extpfeil}           
\usepackage{esint}              

\usepackage{tikz-cd}

\usepackage{microtype}          

\usepackage{geometry}
\usepackage{enumitem}
\setlist[enumerate]{label=\alph*), itemsep=.5em, topsep=2pt}

\usepackage{aliascnt}
\usepackage[unicode]{hyperref}
    \hypersetup{pdfborder = {0 0 0}}
\usepackage[nameinlink]{cleveref}

\allowdisplaybreaks                 

\numberwithin{equation}{section}
\numberwithin{figure}{section}

\usepackage[framemethod=TikZ]{mdframed}
\mdfdefinestyle{thmstyle}{
    backgroundcolor=gray!10,
    hidealllines=true,
    roundcorner=4pt,
    skipabove=\baselineskip,
    skipbelow=0pt,
    innerleftmargin=.5em,
    innerrightmargin=.5em
}

\theoremstyle{definition}
    \newtheorem{definition}{Definition}[section]
    
    \newtheorem*{definition*}{Definition}

\theoremstyle{plain}
    \newtheorem{maintheorem}{Theorem}

    \newaliascnt{theorem}{definition}
    \newtheorem{theorem}[theorem]{Theorem}
    \aliascntresetthe{theorem}
    
    \newtheorem*{theorem*}{Theorem}

    \newaliascnt{proposition}{definition}
    \newtheorem{proposition}[proposition]{Proposition}
    \aliascntresetthe{proposition}
    
    \newtheorem*{proposition*}{Proposition}

    \newaliascnt{corollary}{definition}
    \newtheorem{corollary}[corollary]{Corollary}
    \aliascntresetthe{corollary}
    
    \newtheorem*{corollary*}{Corollary}

    \newaliascnt{lemma}{definition}
    \newtheorem{lemma}[lemma]{Lemma}
    \aliascntresetthe{lemma}
    
    \newtheorem*{lemma*}{Lemma}

\theoremstyle{remark}

\makeatletter
\renewenvironment{proof}[1][\proofname]{\par
  \pushQED{\qed}%
  \normalfont 
  \topsep0pt \partopsep0pt 
  \vspace{-0.5em}
  \trivlist
  \item[\hskip\labelsep
        \itshape
    #1\@addpunct{.}]\ignorespaces
}{%
  \popQED\endtrivlist\@endpefalse
}
\makeatother

\surroundwithmdframed[style=thmstyle]{maintheorem}
\surroundwithmdframed[style=thmstyle]{theorem}
\surroundwithmdframed[style=thmstyle]{lemma}
\surroundwithmdframed[style=thmstyle]{corollary}
\surroundwithmdframed[style=thmstyle]{definition}
\surroundwithmdframed[style=thmstyle]{remark}
\surroundwithmdframed[style=thmstyle]{proposition}

\newcommand{\IR}{\mathbbm{R}}

\newcommand{\IN}{\mathbbm{N}}

\renewcommand{\AA}{\mathfrak{A}}    
\renewcommand{\O}{\mathcal{O}}      
\newcommand{\m}{\mathfrak{m}}       

\renewcommand{\i}{\mathbf{i}}       

\renewcommand{\div}{\operatorname{div}}
\newcommand{\Per}{\operatorname{Per}}

\DeclareMathOperator{\ind}{\varchi}

\DeclareRobustCommand{\varchi}{{\mathpalette\ivarchi\relax}}
\newcommand{\ivarchi}[2]{\raisebox{\depth}{$#1\chi$}}

\DeclareMathOperator{\vol}{vol}
\newcommand{\lowoverline}[1]{\overline{\smash{#1}\raisebox{0.44\ht\strutbox}{}}\vphantom{#1}}
\newcommand{\voln}{\lowoverline{\mathrm{vol}}}

\newcommand{\di}{\mathop{}\mathopen{}\mathrm{d}}

\newcommand{\eps}{\varepsilon}

\newcommand{\filll}{\ \cdot\ }                  
\newcommand{\ubr}[3]{\underbrace{#1}_{#2 \mathrlap{#3}}}
\newcommand{\crit}{{\mathrm{crit}}}             

\DeclarePairedDelimiter{\norm}{\lVert}{\rVert}

\newcommand{\qand}{\quad\text{and}\quad}

\newcommand{\gd}[1]{G_{#1}}           
\newcommand{\bessel}{\mathcal{K}}     

\title[Riemannian Autocorrelation and Non-Local Isoperimetric Energies]{A Riemannian Autocorrelation Function and its Application to Non-Local Isoperimetric Energies}
\date{\today}

\author[M. Bleher]{Michael Bleher}
\address[M. Bleher]{Department of Mathematics,
Heidelberg University,
Im Neuenheimer Feld 205,
D-69120 Heidelberg, Germany}
\email{mbleher@mathi.uni-heidelberg.de}

\author[D. Brazke]{Denis Brazke}
\address[D. Brazke]{Okinawa Institute of Science and Technology, Analysis and Partial Differential Equations Unit, 1919-1 Tancha, Onna-son, 904-0495, Okinawa, Japan.}
\email{denis.brazke@oist.jp}

\author[S. Nill]{Sebastian Nill}
\address[S. Nill]{Department of Mathematics,
Heidelberg University,
Im Neuenheimer Feld 205,
D-69120 Heidelberg, Germany}
\email{snill@mathi.uni-heidelberg.de}

\begin{document}

\markboth{\textsc{M. Bleher, D. Brazke, and S. Nill}}{\textsc{M. Bleher, D. Brazke, and S. Nill}}

\begin{abstract}
We study a family of non-local isoperimetric energies $E_{\gamma,\eps}$ on the round sphere $M = S^n$, where the non-local interaction kernel $K_\eps$ is the fundamental solution of the Helmholtz operator $1 - \eps^2 \Delta$.
To analyse these energies, we introduce a Riemannian autocorrelation function $c_\Omega$ associated to a measurable set $\Omega\subset M$, defined on any compact, connected, oriented Riemannian manifold without boundary $(M^n,g)$ of dimension $n\ge2$.
This function is intimately linked to Matheron's set covariogram from convex geometry.
By establishing a characterisation of functions of bounded variation $BV(M)$ in terms of geodesic difference quotients, we show that $\Omega$ has finite perimeter if and only if $c_\Omega$ is Lipschitz, and we relate the Lipschitz constant to the perimeter of $\Omega$.
We show that on the round sphere $E_{\gamma,\eps}$ admits a reformulation in terms of $c_\Omega$, which allows us to compute the limit as $\eps \to 0$ in a variational sense, that is, in the framework of $\Gamma$-convergence.
\end{abstract}

\maketitle

\section{Introduction}

Let $(M^n,g)$ be a compact, connected, oriented Riemannian manifold without boundary of dimension~$n\ge2$ and $\Omega\subset M$ a measurable subset.
In this article, we are interested in non-local isoperimetric energies of the form
\begin{align}\label{eq:energy_intro}
    E_{\gamma,\eps}(\Omega) \coloneqq \Per(\Omega) - \frac{\gamma}{\eps} \int_{M \times M} K_\eps(x,y) \, |\ind_\Omega(x) - \ind_\Omega(y)| \, \di x \, \di y,
\end{align}
where $\Per(\Omega)$ denotes the variational perimeter (see \Cref{def:total_variation}), $\gamma, \eps > 0$, and $K_\eps$ is a suitable interaction kernel.
This energy provides a sharp interface model for pattern formation in biological membranes and has previously been studied in the flat case \cite{BrKnMC:2023}.
However, many pattern formation phenomena occur not on flat domains, but on curved surfaces (see e.g. the survey \cite{BowGio:2009}).
This motivates generalising the analytical tools developed in these works to study non-local isoperimetric problems on curved geometries.

In the Euclidean case, the non-local term in~\eqref{eq:energy_intro} can be reformulated in terms of the \emph{autocorrelation function} \cite{KnuShi:2023}, defined by
\begin{align}\label{eq:eucl_autocorrelation}
    c_\Omega^{\mathrm{Eucl}}(r) = \frac1{\sigma_{n-1}} \int_{\substack{ w \in \IR^n \\ \norm{w}=1}} \int_{\IR^n} \ind_\Omega(x) \ind_\Omega(x+rw) \, \di x \, \di w ,
\end{align}
where $\sigma_{n-1}$ is the volume of the sphere $S^{n-1} = \{\norm{w}=1\}$.
This function measures the average overlap between $\Omega$ and its translates of distance $r$.
The autocorrelation function is the average of the \emph{covariogram}\footnote{Throughout, we reserve the term autocorrelation for the symmetrised version and covariogram for Matheron's classical object. In the literature, these terms are often used interchangeably.} $C_\Omega^{\text{Eucl}}: h \longmapsto \vol(\Omega \cap (\Omega + h))$ over all translations $h\in\IR^n$ with fixed length $r$.
The covariogram was introduced by Matheron~\cite{Matheron:1986} and is a classical object in stochastic geometry.
For a detailed discussion of the covariogram, we refer to the overview \cite{Bianch:2023}.
Our key insight is that the autocorrelation function admits a generalisation to Riemannian manifolds by replacing translations with the geodesic flow, and that the non-local term in \eqref{eq:energy_intro} can then still be reformulated in terms of this Riemannian autocorrelation function.

Let $\Phi_r$ denote the geodesic flow on the unit tangent bundle $SM$ and write $\widehat{\Omega} \subset SM$ for the preimage of $\Omega \subset M$ under the bundle projection $\pi:SM\to M$.
We define the \emph{Riemannian autocorrelation function} $c_\Omega \colon \IR \longrightarrow \IR$ by
\begin{align}
    c_\Omega(r)
    \coloneqq \frac1{\sigma_{n-1}}\int_{SM} \ind_{\widehat{\Omega}}(\theta) \, \ind_{\widehat{\Omega}}(\Phi_r(\theta)) \, \di\theta,
\end{align}
where the integral is with respect to the Liouville measure.
Equivalently, the autocorrelation function can be written as
\begin{align}
    c_\Omega(r) = \frac1{\sigma_{n-1}}\int_M \int_{\substack{ w \in T_x M \\ \norm{w}_{g_x} = 1 }} \ind_\Omega(x) \ind_\Omega(\exp_x(rw)) \, \di w \, \di x,
\end{align}
which directly generalises the Euclidean formula~\eqref{eq:eucl_autocorrelation}, except that the order of integration has been reversed.

In the Euclidean case, a result due to Galerne \cite{Galerne:2011} establishes that the autocorrelation function encodes the perimeter:
If $\Omega \subset \IR^n$ has finite perimeter, then $c_\Omega^{\mathrm{Eucl}}$ is Lipschitz and its right-sided derivative at $r=0$ is proportional to $\Per(\Omega)$.
Our first main result, which relies on a result of Kreuml and Mordhorst \cite{KreMor:2019}, generalises this characterisation to Riemannian manifolds (this is \Cref{thm:autocorrelation_function_Lipschitz} below).

\begin{maintheorem}\label{mainthm:A}
    Let $(M^n,g)$ be a compact, connected, oriented Riemannian manifold without boundary of dimension $n\ge2$ and $\Omega \subset M$ measurable.
    Then the following are equivalent:
    \begin{enumerate}
        \item $\Omega$ is a set of finite perimeter;
        \item The autocorrelation function $c_\Omega$ is Lipschitz continuous.
    \end{enumerate}
    In that case, the right-sided derivative at $r=0$ is
    \begin{align}
        c'_\Omega(0) = - \frac{k_n}2 \Per(\Omega),
    \end{align}
    and coincides with the Lipschitz constant, i.e. $\|c'_\Omega\|_{L^\infty(\IR_{\ge0})} = \frac{k_n}2 \Per(\Omega)$.
\end{maintheorem}
The proportionality constant $k_n$ is exactly the same as in the flat case and depends only on dimension; the curved geometry enters only through the perimeter $\Per(\Omega)$, which depends on the Riemannian volume form.

As an application, we specialise to the round sphere $M = S^n$ and take $K_\eps$ to be the Helmholtz kernel, i.e. the fundamental solution of the Helmholtz operator $1 - \eps^2\Delta$ (cf. \cite{FHLZ:2016}).
The parameter $\eps$ can be viewed as a screening length controlling the interaction range. 
We study the short-range regime $\eps\to 0$, in which $K_\eps$ concentrates near the diagonal. 
In this regime, the non-local term in \eqref{eq:energy_intro} can be expressed in terms of the Riemannian autocorrelation function $c_\Omega$. 
Combined with \Cref{mainthm:A}, we obtain a uniform Lipschitz control that identifies the leading-order effect of the interaction term.
This leads to a critical interaction strength $\gamma_{\crit}$ and, in particular, to a localisation result in the subcritical regime.
Our second main result makes this localisation precise by identifying the $\Gamma$-limit of $E_{\gamma,\eps}$ as $\eps \to 0$ (this is \Cref{thm:gamma_convergence} below).
\begin{maintheorem}
    Let $M=S^n$ be the round sphere and $0 < \gamma < \gamma_{\crit}$.
    Then $E_{\gamma,\eps}$ $\Gamma$-converges in the $L^1$-topology to the functional
    \begin{align}
        E_{\gamma,0}(\Omega) = \left\{ \begin{array}{lll}
            (1 - \frac{\gamma}{\gamma_{\crit}}) \Per(\Omega)  && \quad \text{if } \Omega \text{ has finite perimeter}, \\[6pt]
            +\infty  && \quad \text{otherwise.}
        \end{array}\right.
    \end{align}
\end{maintheorem}
For the round sphere, we find $\gamma_\crit = 1$, independent of the sphere's size and matching the flat case \cite{BrKnMC:2023}.
In the subcritical regime, the energy localises to a constant multiple of the perimeter functional.
Since minimisers of the perimeter on the round sphere are geodesic balls \cite{Ros:2005}, fine-scale patterns do not form in that case.

\subsection*{Related Literature}
The literature regarding non-local isoperimetric problems is vast and the following list is by no means comprehensive.
For flat underlying geometries, various scenarios of asymptotic expansions in the framework of $\Gamma$--convergence are studied e.g. in \cite{ChoPel:2010, GoMuSe:2013, GoMuSe:2014}. In the context of large mass minimisers of Gamow's liquid drop model, existence of minimisers, stability of minimisers, as well as $\Gamma$--convergence of the non-local isoperimetric energy in the subcritical regime is established in \cite{Pegon:2021} for a large class of interaction kernels $K_\eps$ (see also \cite{MerPeg:2022}). A similar $\Gamma$--convergence result with focus on pattern formation in biological membranes is proved in \cite{BrKnMC:2023} for the flat torus instead of $\mathbbm R^n$. The $\Gamma$--convergence of $E_{\gamma,\eps}$ in the case when the interaction kernel is the solution to fractional Helmholtz equations on open domains with several boundary conditions is considered in \cite{MelWu:2022}. Higher order $\Gamma$--convergence was shown in the recent paper \cite{MuNoSi:2025} where the interaction kernels are given as Yukawa potentials. Further studies of higher order expansions for various types of interaction kernels $K_\eps$ include \cite{MurSim:2019, CesNov:2020, KnuShi:2023}. 

For curved spaces, the literature is much sparser. For the round $2$-sphere, axisymmetrical critical points are studied in \cite{ChToTs:2015} where the interaction kernel is the solution to the Poisson equation. For sufficiently small relative interaction strength $\gamma$, it is shown in \cite{Topalo:2013} that geodesic balls (up to rotation) minimise the non-local isoperimetric energy.

\subsection*{Organisation}
\Cref{sec:riemannian_manifolds_geodesic_flow} collects the necessary preliminaries on the geodesic flow and the invariance of the Liouville measure.
\Cref{sec:bv_on_manifolds} provides relevant results from the theory of functions of bounded variation on Riemannian manifolds and establishes a key relation between the total variation and the mean geodesic difference quotient (\Cref{thm:variation_as_limit_of_difference_quotient}, \Cref{thm:geodesic_variation_Lipschitz}).
\Cref{sec:riemannian_autocorrelation_function} introduces the Riemannian autocorrelation function, its fundamental properties (\Cref{prop:autocorrelation_function}), and the characterisation of sets of finite perimeter  (\Cref{thm:autocorrelation_function_Lipschitz}).
Finally, \cref{sec:application_pattern_formation} presents the application to non-local isoperimetric energies on the sphere and proves the $\Gamma$-convergence result (\Cref{thm:gamma_convergence}).

\subsection*{Acknowledgments} 
We thank Steffen Schmidt and Gabriel Paternain for valuable discussions.
A major part of this work was carried out while DB was affiliated with Heidelberg University, and while MB and DB were funded by the Deutsche Forschungsgemeinschaft (DFG, German Research Foundation, Germany) under Germany’s Excellence Strategy EXC-2181/1-39090098 (the Heidelberg STRUCTURES Cluster of Excellence).
During finalisation of this work, MB was funded by the European Research Council (ERC) under the European Union’s Horizon 2020 research and innovation programme (project PEPS, no. 101071786).

\section{Riemannian manifolds and the geodesic flow}
\label{sec:riemannian_manifolds_geodesic_flow}

Throughout this section $(M^n,g)$ denotes a connected, oriented, complete Riemannian manifold without boundary of dimension $n \geq 2$.
Completeness ensures that geodesics are globally defined, which suffices for the geometric machinery presented in this section.
Later, when discussing functions of bounded variation and studying pattern formation on manifolds, we will additionally require compactness to obtain uniform bounds on various geometric quantities.

Our approach to generalizing the autocorrelation function relies on replacing averaged Euclidean translations with the geodesic flow on the unit tangent bundle $SM$.
More specifically, on Euclidean space, the autocorrelation function relies fundamentally on translation invariance of the Lebesgue measure.
On a general Riemannian manifold, translations do not exist.
Instead, the geodesic flow on the unit tangent bundle, together with the invariance of the Liouville measure under this flow (Proposition~\ref{prop:geo_flow_volume}), provides the correct substitute for translation invariance.
It allows us to define averaged quantities that generalise the Euclidean autocorrelation function to arbitrary manifolds.
This section collects the necessary geometric preliminaries.

\subsection{The unit tangent bundle and averaged integrals}
\label{subsec:SM_and_averages}

For $x \in M$ we write $T_xM$ for the tangent space at $x$ equipped with the inner product $g_x$ and the associated norm $\|v\|_{g_x} = \sqrt{g_x(v,v)}$.
The unit sphere in $T_xM$ is denoted by $S_xM \coloneqq \bigl\{ w \in T_xM : \|w\|_{g_x} = 1 \bigr\}.$

The unit tangent bundle is
\begin{align}
    SM \coloneqq \bigl\{ (x,w) : x \in M,\ w \in S_xM \bigr\} \subset TM,
\end{align}
equipped with the canonical projection
\begin{align}
    \pi \colon SM \longrightarrow M,
    \qquad
    \pi(x,w) = x.
\end{align}
We shall often write $\theta = (x,w) \in SM$.

The Riemannian metric $g$ induces a Riemannian volume form $\vol_M$ on $M$, as well as a canonical volume form $\vol_{SM}$ on $SM$, usually called the Liouville measure.
We work with its fibrewise normalisation
\begin{align}
    \voln_{SM} \coloneqq \frac{1}{\sigma_{n-1}} \, \vol_{SM},
\end{align}
where as before, we write $\sigma_k$ for the $k$--dimensional volume of the Euclidean unit sphere $S^k \subset \IR^{k+1}$.
Integrals with respect to $\voln_{SM}$ are written as
\begin{align}
    \fint_{SM} F(\theta) \,\di \theta
    \coloneqq \frac{1}{\sigma_{n-1}} \int_{SM} F(\theta) \,\di \vol_{SM}(\theta)
\end{align}
for integrable $F \colon SM \longrightarrow \IR$.

For $x \in M$ and an integrable function $\varphi \colon S_xM \longrightarrow \IR$ we use the normalised spherical average
\begin{align}
    \oint_x \varphi(w) \,\di w
    \coloneqq \frac{1}{\sigma_{n-1}} \int_{S_xM} \varphi(w) \,\di S_x(w),
\end{align}
where $\mathrm d S_x$ is the $(n-1)$--dimensional surface measure on $S_xM$ induced by the inner product $g_x$ on $T_xM$.
The next proposition records the disintegration of the Liouville measure along the fibres of~$\pi$.
\begin{proposition}
\label{prop:liouville_disintegration}
    For every $F \in L^1(SM)$,
    \begin{align}
        \fint_{SM} F(\theta) \,\di \theta
        = \int_M \oint_x F(x,w) \,\di w \,\di x.
    \end{align}
\end{proposition}
\begin{proof}
    See e.g. \cite[Theorem VII.1.3.]{Chavel:1995}.
\end{proof}

We will later repeatedly need the average of the absolute value of the Riemannian metric over the unit sphere.
The next proposition shows that this average is proportional to the norm of the underlying vector and identifies the proportionality constant.

\begin{proposition} \label{prop:sphere_integration}
    For $x \in M$ and $u \in T_x M$ we have
    \begin{align}
        \oint_x |g_x(u,w)| \di w = k_n \norm{u}_{g_x},
    \end{align}
    where $k_n$ is the \emph{sphere volume ratio}, given by
    \begin{align}
        k_n
        \coloneqq \frac{2}{n-1} \frac{\sigma_{n-2}}{\sigma_{n-1}}
        = \frac{1}{\pi} \frac{\sigma_n}{\sigma_{n-1}}
        = \frac{2}{\sigma_{n - 1}}\frac{\pi^{\frac{n-1}2}}{\Gamma(\tfrac{n+1}2)}.
    \end{align}
\end{proposition}
\begin{proof}
Every inner product space is isometric to the standard Euclidean space.
Indeed, the Gram-Schmidt process allows us to construct an orthogonal linear isomorphism
\begin{align}
    \psi: (T_x M, g_x)
    \xrightarrow[\mspace{40mu}]{\cong}
    (\IR^n,\langle \cdot, \cdot \rangle)
    \qquad\text{satisfying}\qquad
    g_x(v_1,v_2) = \langle \psi(v_1), \psi(v_2) \rangle
\end{align} 
for all $v_1,v_2 \in T_x M$.
We set $z \coloneqq \psi(u)$.
Since $\norm{w}_{g_x} = \norm{\psi(w)}$,
the transformation $w \longmapsto \psi(w) = y$ sends the unit sphere in $T_x M$ to the one in $\IR^n$.
Hence, the integral can be rewritten as
\begin{align}
    \int_{\substack{ w \in T_x M \\ \norm{w}_{g_x} = 1 }} |g_x(u,w)| \di w
    = \int_{\substack{ y \in \IR^n \\ \norm{y} = 1 }}
        |\langle y,z \rangle| \di y.
\end{align} 
As rotations act transitively on $S^{n-1}$,
we may choose $R \in SO(n)$ with $Rz = \norm{z} e_1$.
Then
\begin{align}
    \langle y,z \rangle
    = \langle Ry,Rz \rangle
    = \norm{z} \langle Ry,e_1 \rangle
    = \norm{u}_{g_x} \langle x,e_1 \rangle
    = \norm{u}_{g_x} x_1,
\end{align}
where we set $x \coloneqq Ry$ and use that $\norm{z} = \norm{\psi(u)} = \norm{u}_{g_x}$.
Thus,
\begin{align}
    \int_{\substack{ y \in \IR^n \\ \norm{y} = 1 }}
        |\langle y,z \rangle| \di y
    = \norm{u}_{g_x} \int_{\substack{ x \in \IR^n \\ \norm{x} = 1 }} |x_1| \di x.
\end{align}
Using poly-spherical coordinates to evaluate the remaining integral, we obtain
\begin{align}
    \int_{ S^{n-1} } |x_1| \di x
    &= \int_0^\pi \int_{ S^{n-2} }
        |\cos(\phi)| \sin^{n-2}(\phi) \di z \di \phi
\\
    &= 2 \, \sigma_{n-2} \int_0^{\frac{\pi}2}
        \sin^{n-2}(\phi) \ubr{\cos(\phi) \di \phi}{=}{\,\di \sin(\phi)}
\\[-2ex]
    &= 2 \, \sigma_{n-2} \int_0^1
        s^{n-2} \di s
\\
    &= \frac2{n-1} \, \sigma_{n-2}.
\end{align}
Normalising by the sphere volume $\sigma_{n-1}$ yields the claimed identity
\begin{align}
    \oint_x |g_x(u,w)| \di w
    = \frac{1}{\sigma_{n-1}} \cdot \frac{2}{n-1} \, \sigma_{n-2} \cdot \norm{u}_{g_x}
    = k_n \norm{u}_{g_x}.
\end{align}
The alternative expressions for $k_n$ follow from the well-known sphere volume formulas
\begin{align} \label{eq:sphere_volume_formula}
    \sigma_n = \frac{2\pi}{n-1} \sigma_{n-2}
    \qquad\text{and}\qquad
    \sigma_n = \frac{2\pi^\frac{n+1}2}{\Gamma(\frac{n+1}2)}.
\end{align}
\end{proof}

\subsection{The geodesic flow}
\label{subsec:geodesic_flow}

For $(x,v) \in TM$ we denote by $\gamma_{x,v} \colon \IR \longrightarrow M$ the unique geodesic with initial data
\begin{align}
    \gamma_{x,v}(0) = x,
    \qquad
    \dot\gamma_{x,v}(0) = v.
\end{align}
Completeness of $(M,g)$ ensures that $\gamma_{x,v}$ is defined for all $t \in \IR$.
The exponential map at $x$ is given by
\begin{align}
    \exp_x \colon T_xM \longrightarrow M,
    \qquad
    \exp_x(v) \coloneqq \gamma_{x,v}(1).
\end{align}
Equivalently, $\gamma_{x,v}(t) = \exp_x(tv)$.

The geodesic flow is the flow on the tangent bundle $TM$ obtained by following geodesics in phase space.
For $t \in \IR$ define
\begin{align}
    \Phi_t \colon TM \longrightarrow TM,
    \qquad
    \Phi_t(x,v)
    \coloneqq \bigl( \gamma_{x,v}(t), \dot\gamma_{x,v}(t) \bigr).
\end{align}
In terms of the exponential map this can be written as
\begin{align}
    \Phi_t(x,v)
    = \bigl( \exp_x(tv), \di_{tv}\exp_x(v) \bigr).
\end{align}
The flow property $\Phi_{s+t} = \Phi_s \circ \Phi_t$ and $\Phi_0 = \mathrm{id}_{TM}$ is immediate from the uniqueness of geodesics.

Since geodesics preserve the speed of their tangent vectors, the geodesic flow preserves the unit tangent bundle:
\begin{align}
    \Phi_t(SM) \subset SM
    \qquad\text{for all } t \in \IR.
\end{align}
We shall use the notation
\begin{align}
    \theta^t
    \coloneqq \Phi_t(\theta)
    \qquad\text{for } \theta \in SM.
\end{align}

The following invariance property is the cornerstone of our approach.
It plays exactly the role that translation invariance plays in Euclidean geometry, and will be used throughout the article.

\clearpage 

\begin{proposition}
\label{prop:geo_flow_volume}
    The normalised Liouville measure is invariant under the flow.
    In particular, the following statements hold:
    \begin{enumerate}
        \item 
        Let $G \in L^1(SM)$. Then for every $t \in \IR$,
        \begin{align}
            \fint_{SM} G(\theta) \di\theta
            = \fint_{SM} G(\theta^t) \di\theta.
        \end{align}
        \item Let $F \in L^1(SM \times SM)$.
        Then for all $r,s,t \in \IR$,
        \begin{align}
            \fint_{SM} F(\theta^r,\theta^s) \di\theta
            = \fint_{SM} F(\theta^{r+t},\theta^{s+t}) \di\theta.
        \end{align}
        \item Let $f \colon M \times M \longrightarrow \IR$ be measurable.
        For all $r,s,t \in \IR$,
        \begin{align} 
            \int_M \oint_x f\bigl( \exp_x(rw),\exp_x(sw) \bigr) \di w \di x
            &= \int_M \oint_x f\bigl( \exp_x((r+t)w),\exp_x((s+t)w) \bigr) \di w \di x.
        \end{align}
    \end{enumerate}
\end{proposition}

\begin{proof}
    The Liouville invariance in (a) is classical and can be proved either in local canonical coordinates on $T^*M$ using the Hamiltonian description of geodesic flow or directly in $(TM,g)$ using the Sasaki metric and divergence-free properties of the geodesic vector field.
    We refer to \cite[Chapter~1]{Paternain:1999} for details.

    Statement (b) is obtained from (a) by composing $G$ with $(\theta,\eta) \longmapsto F(\theta,\eta)$ and using the fact that $\Phi_t$ acts diagonally on the two arguments.
    Statement (c) follows from (b) by applying Lemma~\ref{prop:liouville_disintegration} with
    \begin{align}
        F(\theta^r,\theta^s)
        = f\bigl( \pi(\theta^r),\pi(\theta^s) \bigr)
        = f\bigl( \exp_x(rw),\exp_x(sw) \bigr),
    \end{align}
    where $\theta = (x,w) \in SM$.
\end{proof}

A frequently used special case of Proposition~\ref{prop:liouville_disintegration} is obtained by taking a function on $M$ and lifting it constantly along the fibres of $SM$.

\begin{proposition}
\label{prop:constant_lift}
    Let $f \in L^1(M)$ and $\hat f \colon SM \longrightarrow \IR$ be the constant lift $\hat f(x,w) \coloneqq f(x)$.
    Then
    \begin{align}
        \fint_{SM} \hat f(\theta) \di\theta
        = \int_M f(x) \di x,
    \end{align}
    and for every $r \in \IR$,
    \begin{align}
        \fint_{SM} \bigl| \hat f(\theta^r) - \hat f(\theta) \bigr| \di\theta
        = \int_M \oint_x \bigl| f(\exp_x(rw)) - f(x) \bigr| \di w \di x.
    \end{align}
\end{proposition}

\begin{proof}
    The first identity is \Cref{prop:liouville_disintegration} applied to the constant lift.
    For the second identity, apply \Cref{prop:liouville_disintegration} to $F(\theta) = |\hat f(\theta^r) - \hat f(\theta)|$ and use the definition of the geodesic flow.
\end{proof}

\subsection{The Jacobian of the exponential map}
\label{subsec:jacobian_of_exponential_map}

We will later need quantitative control of the Jacobian of the exponential map in Riemannian normal coordinates.
Fix $x\in M$ and consider the exponential map $\exp_x$ defined on a neighbourhood $U\subset T_xM$ of the origin.
The Jacobian determinant $J_x\colon U\to\IR$ is defined as the proportionality factor between the pullback of the Riemannian volume form on $M$ and the Euclidean volume form on $T_xM$, that is
\begin{align} \label{eq:jacobian_transformation_formula}
    (\exp_x^\ast \vol_M)(v)
    = J_x(v)\,\vol_{\mathrm{eucl}},
    \qquad v\in U,
\end{align}
where $\vol_{\mathrm{eucl}}$ is the Euclidean volume form on $T_xM\simeq \IR^n$ induced by $g_x$, which coincides with the usual Lebesgue measure on $\IR^n$.

Let $y\in M$ and let $0<R$ be smaller than the injectivity radius at $y$.
Then the change-of-variables formula to polar normal coordinates reads
\begin{align}\label{eq:polar-coordinates-1}
    \int_{B_R(y)} f(x) \di x = \sigma_{n-1} \int_0^R \oint_y f(\exp_y(rw)) J_y(rw) r^{n-1} \di w \di r .
\end{align}
In particular, for the round sphere $M=S^n$ of radius 1 the Jacobian determinant of the exponential map is given by $J_y(rw) = \big(\frac{\sin(r)}{r} \big)^{n - 1}$ and we obtain
\begin{align}\label{eq:polar-coordinates-2}
    \int_{S^n \setminus \{\pm y\}} f(x) \di x = \sigma_{n-1} \int_{0}^\pi \oint_y f(\exp_y(rw)) \di w\ \sin(r)^{n-1} \di r .
\end{align}

The following proposition records the asymptotic expansion we will use later in \autoref{prop:lower_bound_of_upper_limit_of_Q_Mordhorst}.
\begin{proposition}
\label{prop:jacobian_expansion}
    Let $x\in M$ and $w\in S_xM$.
    Then, as $r\to 0$, the Jacobian determinant of the exponential map satisfies
    \begin{align}
        J_x(rw)
        = 1 - \frac{1}{6}\,\operatorname{Ric}_x(w,w)\,r^2 + \mathcal O(r^3),
    \end{align}
    where $\operatorname{Ric}$ denotes the Ricci curvature.
\end{proposition}
\begin{proof}
    See, for example, \cite[Chapter~XV, Corollary~3.3]{Lang:1999}.
\end{proof}

\begin{corollary}
\label{cor:jacobian_exp_bound}
    Assume that $M$ is compact.
    Then there exist constants $r_0>0$ and $C>0$ such that for all $x\in M$, all $w\in S_xM$, and all $r\in[0,r_0]$,
    \begin{align}
        J_x(rw) \le 1 + C r^2.
    \end{align}
\end{corollary}
\begin{proof}
    Since $M$ is compact, the Ricci curvature is bounded on $SM$.
    Thus, the $O(r^3)$ term in \Cref{prop:jacobian_expansion} can be chosen uniformly in $(x,w)\in SM$, which yields the existence of $r_0>0$ and $C>0$ with the stated inequality.
\end{proof}

\section{Functions of Bounded Variation on Riemannian Manifolds}
\label{sec:bv_on_manifolds}

Throughout this section $(M^n, g)$ is a compact, connected, oriented Riemannian manifold without boundary of dimension $n\ge2$, equipped with the volume measure $\vol_g$ induced by the metric $g$.
We first recall the definition and relevant properties of functions of bounded variation on $M$, following \cite{MPPP-2007, KreMor:2019}.
We then introduce two key intermediate objects: the \emph{mean geodesic variation} $G_f(r)$, which measures the average $L^1$-change of $f$ along geodesics of length $r$, and the associated \emph{mean geodesic difference quotient} $Q_f(r) = G_f(r)/r$.
The main result (Theorem~\ref{thm:variation_as_limit_of_difference_quotient}) establishes that $\lim_{r\to0^+} Q_f(r) = k_n V[f]$, providing a characterisation of $BV(M)$ entirely in terms of geodesic behaviour.
The proof proceeds in three steps:
We first establish an upper bound via smooth approximation (\Cref{subsec:upper_bound}), then a matching lower bound via radial mollifiers (\Cref{subsec:lower_bound}), and finally combine these to obtain the limit (\Cref{subsec:bv_limit}).

\subsection{Basic definitions and properties}

Denote by $\Gamma(TM)$ the space of smooth vector fields on $M$ and write $\div(X)$ for the divergence of a vector field $X$ with respect to $g$.

\begin{definition}
    \label{def:total_variation}
    Let $f \in L^1(M)$. The \emph{total variation} of $f$ is defined as
    \begin{align}
        V[f] \coloneqq \sup \left\{ \int_M f \, \div(X) \,\mathrm{d}\vol_M : X \in \Gamma(TM), \, \|X(p)\|_{g_x} \leq 1 \text{ for all } x \in M \right\}.
    \end{align}
    We say $f \in L^1(M)$ has \emph{bounded variation} and write $f \in BV(M)$, if $V[f] < \infty$.
    We call $\Omega \subset M$ a set of finite perimeter if its characteristic function $\ind_\Omega \in BV(M)$.
    In that case we set $\Per(\Omega) \coloneqq V[\ind_\Omega]$.
\end{definition}

For differentiable functions $f \in C^1(M)$, the total variation coincides with the $L^1$-norm of the function's gradient $V[f] = \norm{\nabla f}_{L^1(M)}$ (see e.g. \cite[p.454]{KreMor:2019}).
Relatedly, for domains $\Omega$ with smooth boundary $\partial\Omega$, the perimeter coincides with the $(n-1)$-dimensional Riemannian volume of the boundary.

The space $BV(M)$ becomes a Banach space when equipped with the norm
\begin{align}
    \norm{f}_{BV(M)} \coloneqq \norm{f}_{L^1(M)} + V[f], \quad f \in BV(M).
\end{align}

The upcoming results record standard density and compactness properties of $BV(M)$ that will be used throughout this article.

\begin{proposition} \label{prop:bv_lower_semicontinuous}
    Let $f_k \in BV(M)$ and $f \in L^1(M)$ such that $f_k \to f$ in $L^1(M)$. Then
        \begin{align}
            V[f] \leq \liminf_{k \to \infty} V[f_k].
        \end{align}
\end{proposition}
\begin{proof}
    Let $X \in \Gamma(TM)$ such that $\norm{X(p)}_{g(p)} \leq 1$ for all $x \in M$.
    It follows from the assumed $L^1$--convergence that
    \begin{align}
        \int_M f \, \div(X) \di \mathrm{vol}_g = \lim_{k \to \infty}\int_M f_k \, \div(X) \di \mathrm{vol}_g \leq \liminf_{k \to \infty} V[f_k].
    \end{align}
    Taking the supremum over all vector fields with $\norm{X(p)}_{g(p)} \leq 1$ proves the claim.
\end{proof}

\begin{proposition} \label{prop:smooth_functions_dense_in_bv}
    Then for every $f \in BV(M)$ and for every $\eps > 0$ there exists $\varphi \in C^\infty(M)$ such that $\|f - \varphi \|_{L^1(M)} + \bigl| V[f] - V[\varphi] \bigr| < \eps$.
\end{proposition}
\begin{proof}
See \cite[Proposition 1.4]{MPPP-2007}.
\end{proof}

\begin{proposition} \label{prop:bv_compact}
    Let $f_k \in BV(M)$ such that $\sup_k \norm{f_k}_{BV(M)} < \infty$.
    Then there exists $f \in BV(M)$ and a subsequence (not relabeled) such that $f_k \to f$ in $L^1(M)$.
\end{proposition}
\begin{proof} 
    Using \Cref{prop:smooth_functions_dense_in_bv} we find a sequence $\varphi_k \in C^\infty(M)$ such that
    \begin{align} \label{eq:compactness_equation_closeness}
        \norm{f_k - \varphi_k}_{L^1(M)} + \bigl| V[f_k] - V[\varphi_k] \bigr| < \frac 1k\ , \quad k \in \IN.
    \end{align}
    Thus $\varphi_k$ is a bounded sequence in the Sobolev space $W^{1,1}(M)$. The Rellich-Kondrachov Theorem (see \cite[Theorem 2.9]{Hebey:1999}) implies that there exists $f \in L^1(M)$ and a subsequence (not relabeled) such that $\varphi_k \to f$ in $L^1(M)$. Using \eqref{eq:compactness_equation_closeness} we find that $f_k \to f$ in $L^1(M)$. Using \Cref{prop:bv_lower_semicontinuous} we find $f \in BV(M)$, which concludes the proof.
\end{proof}

The following result by Kreuml-Mordhorst provides a way to calculate the variation of $f\in L^1(M)$.
This result is a generalisation of the famous Bourgain-Brezis-Mironsecu formula from the Euclidean setting to Riemannian manifolds (see \cite{Davil:2002, BBM:2001}).
It expresses the total variation as a limit of non-local functionals built using radial mollifiers.

A family of \textit{radial mollifiers} $(\rho_s)_{s\in(0,1)}$ is a set of functions
\begin{align}
    \rho_s \colon \IR_{\ge0} \longrightarrow \IR_{\ge0},\ 0 < s < 1
\end{align}
satisfying the following properties:
\begin{enumerate}
    \item Each $\rho_s$ is monotonically decreasing on $\IR_{\ge0}$,
    \item $\displaystyle \int_0^\infty \rho_s(r) \, r^{n-1} \di r = \frac1{\sigma_{n-1}}$ for all $s \in (0,1)$,
    \item $\displaystyle \lim_{s\to0} \int_\delta^\infty \rho_s(r) \, r^{n-1} \di r = 0$ for all $\delta>0$,
    \item $\displaystyle \lim_{s\to0} \sup_{r\in K} \rho_s(r) = 0$ for all compact $K \subset \IR_{\ge0}$.
\end{enumerate}

\newcommand{\KMThm}{\cite[Theorem 1.1]{KreMor:2019}}
\begin{theorem}[\KMThm]
    \label{thm:mordhorst-kreuml}
    Let $f\in L^1(M)$ and $\rho_s$ a family of radial mollifiers.
    Then the variation of $f$ is given by
    \begin{align}
        \lim_{s\to0} \int_M \int_M \frac{|f(x)-f(y)|}{d(x,y)} \, \rho_s\big( d(x,y) \big) \di x \di y
        = k_n V[f],
    \end{align}
    where $k_n$ is the sphere volume ratio introduced in \Cref{prop:sphere_integration}.
\end{theorem}

\subsection{The mean geodesic variation}

To connect the total variation $V[f]$ with Riemannian geometry, we need a way to measure how $f$ varies along geodesics.
The classical difference quotient approach in $\mathbb{R}^n$ measures $|f(x+h) - f(x)|/|h|$ for small displacements $h$.
On a manifold, we replace the displacement $x \longmapsto x+h$ with the geodesic flow: we follow geodesics of length $r$ from each point $x$ in all directions $w \in S_xM$ and average the resulting variation.
This leads to the following definition.

\begin{definition}
    Let $f \in L^1(M)$.
    The \emph{mean geodesic variation} of $f$ is the function $\gd{f} \colon \IR \longrightarrow \IR$ defined for each $r \in \IR$ by
    \begin{align}
        \gd{f}(r)
        \coloneqq \fint_{SM} \bigl| \hat f(\theta^r) - \hat f(\theta) \bigr| \di \theta,
    \end{align}
    where $\hat f \colon SM \longrightarrow \IR$ is the constant lift of $f$.
\end{definition}

To make the geometric meaning of this definition more apparent, we may rewrite $\gd{f}$ in terms of the exponential map using \Cref{prop:liouville_disintegration}.
For each $r \in \IR$, we have
\begin{align}
    \gd{f}(r)
    = \int_M \oint_x \bigl| f\bigl( \exp_x(rw) \bigr) - f(x) \bigr| \di w \di x .
\end{align}
For fixed $x \in M$, the inner integral measures the mean variation of $f$ along all geodesics of length $r$ emanating from $x$.
The outer integral over $M$ then collects these averages over all base points.

It follows from the definition and the properties of the Liouville measure that for each $r \in \IR$ the map $f \longmapsto \gd{f}(r)$ satisfies the triangle inequality $\gd{f+g}(r) \le \gd{f}(r) + \gd{g}(r)$ and is absolutely homogeneous $\gd{\alpha f}(r) = |\alpha| \gd{f}(r)$.
In other words, $\gd{\bullet}$ is a one-parameter family of seminorms on $L^1(M)$ indexed by $r\in\IR$.
By the same arguments, one moreover obtains the reverse triangle inequality:
\begin{align}
    \label{eq:geodesic_variation_reverse_triangle_inequality}
    |\gd{f} - \gd{g}| \leq G_{f - g},\ f,g \in L^1(M).
\end{align}

For fixed $f$ and as a function of $r$, the mean geodesic variation has the following properties.
\begin{proposition} \label{prop:geodesic_variation_properties}~
    Let $f\in L^1(M)$ and $r,s\in\mathbb{R}$. Then $\gd{f}$ satisfies the following properties:
    \begin{enumerate}
        \item (Uniform bounds) $0 = \gd{f}(0) \leq \gd{f}(r) \leq 2\norm{f}_{L^1(M)}$.
        \item (Reflection invariance) $\gd{f}(-r) = \gd{f}(r)$.
        \item (Subadditivity) $\gd{f}(r+s) \le \gd{f}(r) + \gd{f}(s)$.
        \item (Reverse subadditivity) $\bigl| \gd{f}(r)-\gd{f}(s) \bigr| \le \gd{f}\bigl( |r-s| \bigr)$.
        \item (Continuity) $\gd{f}$ is continuous on $\IR$.
    \end{enumerate}
\end{proposition}

\begin{proof}~
\begin{enumerate}

    \item $\gd{f}(0) = 0$ and $0 \leq \gd{f}(r)$ are immediate from the definition.
    The triangle inequality and the invariance of the Liouville measure (\Cref{prop:geo_flow_volume}) yield
    \begin{align}
        \gd{f}(r) 
        &\leq \fint_{SM} \lvert \hat f(\theta^r ) \rvert \di\theta + \fint_{SM} \lvert \hat f(\theta) \rvert \di\theta
        = 2 \norm{f}_{L^1(M)}.
    \end{align}

    \item The invariance of the Liouville measure under the geodesic flow (\Cref{prop:geo_flow_volume}) implies
    \begin{align}
        \gd{f}(-r) 
        &= \fint_{SM} \lvert \hat f(\theta^{-r}) - \hat f(\theta) \rvert \di \theta
        = \fint_{SM} \lvert \hat f(\theta) - \hat f(\theta^r) \rvert \di \theta
        = \gd{f}(r).
    \end{align}
    
    \item The triangle inequality and the invariance of the Liouville measure (\Cref{prop:geo_flow_volume}) imply 
    \begin{align}
    \begin{multlined}
        \gd{f}(r+s)
        = \fint_{SM} \lvert \hat f(\theta^{r+s}) - \hat f(\theta) \rvert \,\di\theta\\
        \le \fint_{SM} \lvert \hat f(\theta^{r+s}) - \hat f(\theta^s) \rvert \di \theta 
            + \fint_{SM} \lvert \hat f(\theta^{s}) - \hat f(\theta) \rvert \,\di\theta
        = \gd{f}(r) + \gd{f}(s).
    \end{multlined}
    \end{align}

    \item Replacing $r$ by $r-s$ in c) yields $\gd{f}(r)-\gd{f}(s) \le \gd{f}(r-s) = \gd{f}(|r-s|)$,
    where the last step comes from the symmetry proven in b).
    The claim follows after reversing the role of $r$ and $s$.
        
    \item Due to the reverse subadditivity $|\gd{f}(r)-\gd{f}(s)| \le \gd{f}(|r-s|)$ in c), it suffices to show continuity at $r=0$.
    Let $\eps>0$.
    Since $C^0(M)$ is dense in $L^1(M)$, we can find $g \in C^0(M)$ such that $\norm{f - g}_{L^1(M)} \le \tfrac{\eps}{2C}$.
    Since the geodesic flow is smooth, $\gd{g}$ is continuous on $\IR_{\ge0}$.
    In particular, $\gd{g}(r) \to \gd{g}(0) = 0$ for $r\to0$.
    This allows us to pick $r_0>0$ such that $\gd{g}(r) \le \tfrac{\eps}{2}$ for all $0<r<r_0$.
    Together with the reverse triangle inequality of the seminorm $\gd{\bullet}(r)$ in~\eqref{eq:geodesic_variation_reverse_triangle_inequality} and the homogeneous upper bound in d), we obtain the following estimate for $r\in(0,r_0)$:
    \begin{align}
        \begin{multlined}
        \gd{f}(r)
        \le \gd{g}(r) + \bigl| \gd{f}(r) - \gd{g}(r) \bigr|
        \le \gd{g}(r) + \gd{f-g}(r)\\
        \le \gd{g}(r) + C\norm{f-g}_{L^1(M)}
        \le \tfrac{\eps}{2} + C \tfrac{\eps}{2C} = \eps.
        \end{multlined}
    \end{align}

\end{enumerate}
\end{proof}

\subsection{The mean geodesic difference quotient}

The mean geodesic variation $\gd{f}(r)$ measures the average $L^1$--change of $f$ along geodesics.
Normalising $\gd{f}(r)$ by the length $r$ of the geodesics leads to the natural Riemannian analogue of the $L^1$--difference quotients used to describe $BV(\IR^n)$ in terms of finite difference quotients \cite{Galerne:2011}.

\begin{definition}
    Let $f \in L^1(M)$.
    The \emph{mean geodesic difference quotient} of $f$ is the function ${Q_f \colon \IR_{>0} \longrightarrow \IR}$ defined for each $r > 0$ by
    \begin{align}
        Q_f(r)
        \coloneqq \frac{\gd{f}(r)}{r}
        = \fint_{SM} \biggl| \frac{\hat f(\theta^r) - \hat f(\theta)}{r} \biggr| \di \theta ,
    \end{align}
    where $\hat f \colon SM \longrightarrow \IR$ is the constant lift of $f$.
\end{definition}

We can again rewrite $Q_f$ in terms of the exponential map using \Cref{prop:liouville_disintegration}.
More precisely, for each $r > 0$, we have
\begin{align}
    Q_f(r)
    = \int_M \oint_x \biggl| \frac{f\bigl( \exp_x(rw) \bigr) - f(x)}{r} \biggr| \di w \di x ,
\end{align}
so $Q_f$ measures the absolute difference quotient of $f$, averaged over all geodesics of length $r$.

Since $G_f$ is continuous, so is $Q_f$ except possibly in $r = 0$.
In the remaining parts of \cref{sec:bv_on_manifolds}, we analyse $Q_f$ as a function on $\IR_{> 0}$, with particular emphasis on its behaviour as $r \to 0^+$.

\begin{proposition}\label{prop:geodesic_quotient_properties}
    Let $f\in L^1(M)$. Then
    \begin{enumerate}
        \item (Continuity) $Q_f$ is continuous on $\IR_{>0}$.
        \item (Integer Contraction) $Q_f(mr) \le Q_f(r)$ for all $r\in\IR_{>0}$ and $m\in\IN$.
    \end{enumerate}
\end{proposition}

\begin{proof}
a) is a direct consequence of the continuity of $\gd{f}$ (\Cref{prop:geodesic_variation_properties}e).
b) follows by induction on the subadditivity of $\gd{f}$ yields $\gd{f}(mr) \leq m \gd{f}(r)$ for all $m\in\IN$, from which the claim for $Q_f(r)$ follows immediately.
\end{proof}

\begin{proposition} \label{prop:geodesic_quotient_limit}
    Let $f\in L^1(M)$.
    Then the right-sided lower and upper limit of $Q_f$ for $r\to0^+$ agree:
    \begin{align}
        \liminf_{r\to0^+} Q_f(r) &= \limsup_{r\to0^+} Q_f(r) = \sup_{r>0} Q_f(r).
    \end{align}
\end{proposition}

\begin{proof}
    It suffices to show that $Q_f(r) \le \liminf_{s\to0} Q_f(s)$ for all $r\in\IR_{>0}$.
    Let $r\in\IR_{>0}$ and $\eps>0$.
    Since $Q_f$ is continuous at $r$, there is $\delta>0$ such that $|Q_f(r)-Q_f(s)| < \eps$ for all $s\in\IR_{>0}$ with $|r-s| < \delta$.
    Since the limit inferior is an accumulation point, we find $t\in\IR_{>0}$ with $t<\delta$ such that $|Q_f(t) - \liminf_{s\to0} Q_f(s)| < \eps$.
    Now, we set $m \coloneqq \lceil \frac{r}{t} \rceil \in \IN$, defined as the unique natural number $m$ such that $m-1 < \frac{r}{t} \le m$.
    After reshuffling the terms, we obtain $0 \le mt-r < t$ and in particular $|mt-r| < \delta$.
    Therefore
    \begin{align}
        Q_f(r) &\le \eps + Q_f(mt)
        \le \eps + Q_f(t)
        \le 2\eps + \liminf_{s\to0} Q_f(s).
    \end{align}
    The first inequality follows from the continuity of $Q$ at $r$, the second from \Cref{prop:geodesic_quotient_properties}b), and the third from the definition of limes inferior and the choice of $t$.
    As this is true for all $\eps>0$, the claim follows.
\end{proof}

\subsection{Upper bound via smooth approximation}
\label{subsec:upper_bound}

It turns out that the mean geodesic difference quotient $Q_f$ is bounded from above by a multiple of the total variation of $f$ (see also \cite[Proposition 11]{Galerne:2011}).
In the smooth setting, the increment $f(\exp_x(rw)) - f(x)$ along a geodesic segment of length $r$ can be written as an integral of the differential of $f$ along the geodesic flow, and the invariance of the Liouville measure then yields a bound in terms of the $L^1$--norm of $\nabla f$, uniform in $r>0$.
This gives a pointwise inequality $Q_f(r) \le k_n V[f]$ for $C^1$--functions.
Using the density of smooth functions in $BV(M)$ (\Cref{prop:smooth_functions_dense_in_bv}), we can then transfer this estimate to general $BV$--functions and obtain an upper bound on $\limsup_{r\to0^+} Q_f(r)$ in terms of the total variation. 

\begin{lemma} \label{prop:upper_bound_of_upper_limit_of_Q}
    Let $f \in L^1(M)$, then for all $r>0$,
    \begin{align}
        Q_f(r) \leq k_n V[f].
    \end{align}
\end{lemma}
\begin{proof}
    If $f$ does not have bounded variation, the inequality is trivially satisfied.
    For $f \in BV(M)$, we prove the bound by smooth approximation.

    \textit{Step 1: The smooth case.}
    Let $f \in C^1(M)$ and fix $x \in M$, $w \in T_x M$ and $r>0$.
    We use the geodesic path $\gamma \colon [0,1] \longrightarrow M$ given by $\gamma(t) = \exp_x(trw)$ and the fundamental theorem of calculus to find
    \begin{align} \label{eq:hadamard-trick}
        f\bigl( \exp_x(rw) \bigr) - f(x)
        &= \int_0^1 \frac{\mathrm d}{\mathrm dt} \bigl( f \circ \gamma \bigr)(t) \di t.
    \end{align}
    We evaluate the derivative with the chain rule:
    \begin{align}
        \frac{\mathrm d}{\mathrm dt}\bigl( f \circ \gamma \bigr)(t)
        &= \di_{\gamma(t)}f ( \dot\gamma(t) )
        = \bigl( \di_{\exp_x(trw)}f \circ \di_{trw} \exp_x \bigr)(rw)
        = r \, \bigl( \di f \circ \Phi_{tr} \bigr)(x,w),
    \end{align}
    where $\Phi$ is the geodesic flow on the unit tangent bundle $SM$.
    Consequently, we can estimate
    \begin{align}
        Q_f(r)
        &= \frac 1r \int_M \oint_x \bigl| f\bigl( \exp_x(rw) \bigr) - f(x) \bigr| \di w \di x
    \\
        &\le \frac 1r \int_M \oint_x
        \int_0^1 \Bigl| \frac{\mathrm d}{\mathrm dt}\bigl( f \circ \gamma \bigr)(t) \Bigr| \di t \di w \di x
    \\
        &= \int_0^1 \int_M \oint_x
        \bigl| ( \di f \circ \Phi_{tr} )(x,w) \bigr|
        \di w \di x \di t
    \\
        &\overset{\text{(1)}}{=}
        \int_0^1 \int_M \oint_x
        \bigl| \di f(x,w) \bigr|
        \di w \di x \di t
    \\
        &
        \overset{\text{(2)}}{=}
        \int_0^1 \di t \int_M \oint_x
        \bigl| g_x\bigl( \nabla f(x), w \bigr) \bigr|
        \di w \di x
    \\
        &\overset{\text{(3)}}{=} k_n \int_M \norm{\nabla f(x)}_{g_x} \di x 
        \; = \; k_n \, V[f].
    \end{align}
    Equality (1) is the invariance of the Liouville measure under the geodesic flow $\Phi$ (\Cref{prop:geo_flow_volume}).
    In equality (2), we first use the definition of the gradient.
    Then we note that the integrand is now independent of $t$, so we can perform the integration over $[0,1]$.
    Finally, \Cref{prop:sphere_integration} yields equality~(3).

    \textit{Step 2: The general case.}
    Let $f \in BV(M)$, $\eps > 0$ and $r>0$.
    By \Cref{prop:smooth_functions_dense_in_bv}, we can find a smooth function $\varphi \in C^\infty(M)$ such that
    \begin{align}
        \norm{f-\varphi}_{L^1(M)} < \frac{r}4 \eps
        \qand
        \bigl| V[f] - V[\varphi] \bigr| < \frac{1}{2k_n} \eps,
        \quad\text{in particular }
        k_n V[\varphi] < \frac\eps2 + k_n V[f].
    \end{align}
    We then manipulate $Q_f(r)$ by adding and subtracting $\hat\varphi(\theta)$ and applying the triangle inequality twice:
    \begin{align}
        Q_f(r)
        &= \frac1r \fint_{SM} \bigl\lvert \hat f(\theta^r) - \hat \varphi(\theta^r) + \hat \varphi(\theta^r) - \hat \varphi(\theta) + \hat \varphi(\theta) - \hat f(\theta) \bigr\rvert \di\theta
    \\
        &\le \frac1r \fint_{SM} \bigl\lvert \hat f(\theta^r) - \hat \varphi(\theta^r) \bigr\rvert \di \theta
        + \frac1r \fint_{SM} \bigl\lvert \hat \varphi(\theta^r) - \hat \varphi(\theta) \bigr\rvert \di \theta
        + \frac1r \fint_{SM} \bigl\lvert \hat \varphi(\theta) - \hat f(\theta) \bigr\rvert \di\theta
    \\
        &\overset{\text{(1)}}{=} \frac2r \, \norm{f-\varphi}_{L^1(M)}
        \; + \; Q_\varphi(r)
    \\
        &\overset{\text{(2)}}{\le} \frac2r \, \norm{f-\varphi}_{L^1(M)} + k_n V[\varphi]
    \\
        &\overset{\text{(3)}}{\le} \eps + k_n V[f].
    \end{align}
    Here, the invariance of the Liouville measure under the geodesic flow (\Cref{prop:geo_flow_volume}) provides us with equality (1).
    Inequality (2) employs the smooth case proven in Step~1.
    Finally, the choice of $\varphi$ yields inequality (3).
    Since $\eps > 0$ was arbitrary, this is the desired result.
\end{proof}

\subsection{Lower bound via radial mollifiers}
\label{subsec:lower_bound}

We will now show that, as $r\to0^+$, the upper limit of the mean geodesic difference quotient is bounded from below by $k_n V[f]$.
This is a consequence of the characterisation of $BV(M)$ due to Kreuml and Mordhorst \cite{KreMor:2019}.
By choosing a particularly simple family of mollifiers supported in small geodesic balls, and rewriting the corresponding double integral on $M \times M$ in polar normal coordinates, we obtain an average of the geodesic difference quotients $Q_f(r)$ over small radii.
Uniform control on the Jacobian of the exponential map then allows us to pass to the limit and deduce the desired bound.

\begin{lemma} \label{prop:lower_bound_of_upper_limit_of_Q_Mordhorst}
    Let $f \in L^1(M)$, then
    \begin{align}
        k_n V[f] &\le \limsup_{r\to0} Q_f(r).
    \end{align}
\end{lemma}

\begin{proof}
The following set of piecewise-continuous, bounded functions provides a family of radial mollifier:
\begin{align}
    \rho_s(r) = \frac1{\sigma_{n-1} \, s \, r^{n-1}} \ind_{[0,s]}(r).
\end{align}
By the result of Kreuml and Mordhorst (\Cref{thm:mordhorst-kreuml}), the total variation of $f$ is given by
\allowdisplaybreaks
\begin{align}
    k_n V[f]
    &= \lim_{s\to0} \int_M \int_M \frac{|f(x)-f(y)|}{d(x,y)} \, \rho_s\big( d(x,y) \big) \di x \di y
\\[6pt]
    &\overset{\text{(1)}}{=} \lim_{s\to0} \frac1s \int_M \frac1{\sigma_{n-1}} \int_{B_s(x)} \frac{|f(y)-f(x)|}{d(x,y)^n} \di y \di x
\\[6pt]
    &\overset{\text{(2)}}{=} \lim_{s\to0} \frac1s \int_M \int_0^s \oint_x \frac{|f(\exp_x(rw))-f(x)|}{r^n} \, J_x(rw) \, r^{n-1} \di w \di r \di x
\\[6pt]
    &\overset{\text{(3)}}{\le} \limsup_{s\to0} \fint_0^s \int_M \oint_x \frac{|f(\exp_x(rw))-f(x)|}{r} (1+Cr^2) \di w \di x \di r
\\[6pt]
    & \leq \limsup_{s \to 0} \sup_{0 < r < s} \bigg\{ \int_M \oint_x \frac{|f(\exp_x(rw))-f(x)|}{r} \di w \di x \bigg\} \underbrace{\fint_0^s (1 + Cr^2) \di r}_{= 1 + \frac C3 s^2}
\\[6pt]
    & = \limsup_{r\to0} \int_M \oint_x \frac{|f(\exp_x(rw))-f(x)|}{r} \di w \di x
\\[6pt]
    & = \limsup_{r\to0} Q(r).
\end{align}
Equality~(1) uses the definition of $\rho_s$, equality~(2) expresses the integration in polar normal coordinates using the formula~\eqref{eq:polar-coordinates-1}, and inequality~(3) uses that $M$ is compact, such that the bound of \Cref{cor:jacobian_exp_bound} holds.
\end{proof}

\subsection{Total variation as limit of the geodesic difference quotient}
\label{subsec:bv_limit}

We are now in the position to relate the mean geodesic difference quotient $Q_f$ to the total variation $V[f]$.
The next result shows that, for every $f \in L^1(M)$, the right-sided limit $\lim_{r\to0^+} Q_f(r)$ coincides with $k_n V[f]$ (possibly $+\infty$).

\begin{theorem} \label{thm:variation_as_limit_of_difference_quotient}
    Let $f \in L^1(M)$, then
    \begin{align}
        \lim_{r\to0^+} Q_f(r) = k_n V[f].
    \end{align}
\end{theorem}

\begin{proof}
Since the lower and upper limit agree according to \Cref{prop:geodesic_quotient_limit}, the claim follows by combining the pointwise upper bound in \Cref{prop:upper_bound_of_upper_limit_of_Q} and the lower bound on the limsup in 
\Cref{prop:lower_bound_of_upper_limit_of_Q_Mordhorst}.
\end{proof}

An immediate consequence of Theorem~\ref{thm:variation_as_limit_of_difference_quotient} is a characterisation of the space $BV(M)$ in terms of the mean geodesic variation $G_f$:

\begin{theorem} \label{thm:geodesic_variation_Lipschitz}
    Let $f \in L^1(M)$.
    Then the following statements are equivalent:
    \begin{enumerate}
        \item $f \in BV(M)$.
        \item $\gd{f}$ is Lipschitz continuous.
    \end{enumerate}
    In that case, the right-sided derivative of $\gd{f}$ in $r=0$ exists and takes the value
    \begin{align} \label{eq:thm:derivative_mean_difference}
        G'_f(0) = \|G'_f\|_{L^\infty(\IR)} = k_n V[f].
    \end{align}
\end{theorem}

\begin{proof}
    Assume that $f \in BV(M)$. We use \Cref{prop:geodesic_variation_properties}c), \Cref{prop:geodesic_quotient_limit} and \Cref{thm:variation_as_limit_of_difference_quotient} to obtain
    \begin{align}
        \bigl| G_f(r) - G_f(s) \bigr| \leq G_f\bigl( |r - s| \bigr) \leq |r - s| \, \sup_{|r - s| > 0} Q_f\bigl( |r - s| \bigr) = k_n V[f] \, |r - s|
    \end{align}
    for all $r,s \in \IR$. Thus $G_f$ is Lipschitz continuous.
    
    Assume now that $\gd{f}$ is Lipschitz continuous. Then we find $L > 0$ such that
    \begin{align}
        |G_f(r) - G_f(s)| \leq L \, |r - s| && \text{for all } r,s \in \IR.
    \end{align}
    Choosing $s = 0$ and using \Cref{thm:variation_as_limit_of_difference_quotient} we find that
    \begin{align}
        k_n V[f] = \limsup_{r \to 0^+} \frac{G_f(r)}{r} \leq L,
    \end{align}
    which shows that $f \in BV(M)$. The formula for the right-sided derivative now follows from \Cref{thm:variation_as_limit_of_difference_quotient}.
    
    To obtain the $L^\infty$-estimate, we note that since $G_f$ is Lipschitz, it is also differentiable almost everywhere by Rademacher's theorem. Let $r > 0$ such that $G_f$ is differentiable in $r$. Then using the pre-Lipschitz property in \Cref{prop:geodesic_variation_properties} we find that
    \begin{align}
        |G_f'(r)| = \lim_{h \to 0^+} \frac{|G_f(r + h) - G_f(r)|}{h} \leq \liminf_{h \to 0^+} \frac{G_f(h)}{h} = k_n \, V[f],
    \end{align}
    which concludes the proof.
\end{proof}

\section{The Riemannian Autocorrelation Function}
\label{sec:riemannian_autocorrelation_function}

In this section, we introduce and study the autocorrelation function on compact, connected, oriented Riemannian manifolds without boundary $(M^n,g)$ of dimension $n \geq 2$.
Using the theory developed in \cref{sec:bv_on_manifolds}, we show that the autocorrelation function defined on manifolds shares the key properties of its Euclidean counterpart.

The Euclidean autocorrelation function $c_\Omega^{\mathrm{Eucl}}(r)$ in~\eqref{eq:eucl_autocorrelation} arises by averaging the covariogram $C_\Omega^{\text{Eucl}}(h) = \vol(\Omega \cap (\Omega + h))$ over all translations $h\in\mathbb{R}^n$ of length $r$.
On a general manifold, however, there is no global notion of translation, and hence no canonical covariogram to average.
One natural approach is to generalise the covariogram to homogeneous spaces, i.e. manifolds that are equipped with a transitive group action that can replace the notion of translations.
But the existence of such an action is a restrictive assumption that may exclude many geometrically interesting manifolds.
We instead bypass the covariogram entirely and define the autocorrelation function essentially by changing the order of integration: first average over geodesics of length $r$ emanating in all directions, then integrate over all starting points.
This is naturally realised by following the geodesic flow on the unit tangent bundle $SM$ and integrating with respect to the Liouville measure, which captures the simultaneous average over all directions and all starting points.

\begin{definition}\label{def:riemannian_autocorrelation}
    Let $\Omega \subset M$ be measurable and let $\widehat \Omega \coloneqq \pi^{-1}(\Omega)$ be its preimage with respect to the canonical projection $\pi \colon SM \longrightarrow M$.
    We define the \emph{autocorrelation function} $c_\Omega \colon \IR \longrightarrow \IR$ for all $r \in \IR$ via
    \begin{align}\label{eq:riemannian_autocorrelation_SM}
        c_\Omega(r) 
        &\coloneqq \fint_{SM}\mspace{-7mu} \ind_{\widehat\Omega}(\theta) \ind_{\widehat\Omega}(\theta^r) \di\theta,
    \end{align}
    or equivalently,
    \begin{align}\label{eq:riemannian_autocorrelation_M}
        c_\Omega(r)
        &= \int_M \oint_x \ind_\Omega(x) \ind_\Omega\bigl( \exp_x(rw) \bigr) \di w \di x.
    \end{align}
\end{definition}

The autocorrelation function $c_\Omega(r)$ captures all the geometric information about how $\Omega$ intersects with itself, on average, under the geodesic flow.
More precisely, the intersection between any two geodesic translates $\Phi_r \widehat{\Omega}$ and $\Phi_s \widehat{\Omega}$ depends only on their relative geodesic displacement $r-s$, not on the individual parameters $r$ and $s$.
\begin{proposition} \label{prop:double_single_autocorrelation_fct}
    Let $\Omega \subset M$ be measurable.
    Then for all $r,s \in \IR$
    \begin{align}
        \voln\bigl( \Phi_r \widehat{\Omega} \cap \Phi_s \widehat{\Omega} \bigr) = c_\Omega(r-s).
    \end{align}
    In particular, the value of the autocorrelation function at $r$ is given by $c_\Omega(r) = \voln\bigl( \widehat{\Omega} \cap \Phi_r \widehat{\Omega} \bigr)$.
\end{proposition}
\begin{proof}
    By definition we have
    \begin{align}
        \voln\bigl( \Phi_r \widehat{\Omega} \cap \Phi_s \widehat{\Omega} \bigr) 
        = \fint_{SM} \ind_{\widehat\Omega}(\theta^r) \ind_{\widehat\Omega}(\theta^{s}) \di\theta
        = \fint_{SM} \ind_{\widehat\Omega}(\theta^{r-s}) \ind_{\widehat\Omega}(\theta) \di\theta
        = c_\Omega(r-s),
    \end{align}
    where the second equality is exactly the invariance of the Liouville measure under the geodesic flow described in \Cref{prop:geo_flow_volume}, while the final equality is again by definition.
\end{proof}

The following formula, analogous to Matheron's classical result \cite{Matheron:1986} (see also \cite[Lemma 10]{Galerne:2011}), is the crucial link that connects the autocorrelation function $c_\Omega$ to the mean geodesic variation $G_{\chi_\Omega}$.
This relationship enables us to translate properties of $G_f$ into properties of $c_\Omega$ and ultimately to characterise sets of finite perimeter via the autocorrelation function.

\begin{lemma} \label{lemma:matheron}
    Let $r \in \IR$ and $\Omega \subset M$ be measurable.
    Then
    \begin{align}
        c_\Omega(0) - c_\Omega(r) = \tfrac12 \gd{\ind_\Omega}(r).
    \end{align}
\end{lemma}

\begin{proof}
Since the difference of two characteristic functions only takes the values $\{0,\pm 1\}$, we find that
\begin{align}
    \left| \ind_\Omega(q) - \ind_\Omega(q') \right|
    &= \left( \ind_\Omega(q) - \ind_\Omega(q') \right)^2
    = \ind_\Omega(q)^2 - 2 \ind_\Omega(q) \ind_\Omega(q') + \ind_\Omega(q')^2
\end{align}
Setting $q' \coloneqq x$ and $q \coloneqq \exp_x(rw)$ allows us to rewrite the upper integral as follows:
\begin{align}
    \int_M \oint_x
    \bigl| \ind_\Omega\bigl( \exp_x(rw) \bigr) - \ind_\Omega(x) \bigr| \di w \di x
    &= c_\Omega(0) - 2 \, c_\Omega(r) + c_\Omega(0)
\\
    &= 2 \, c_\Omega(0) - 2 \, c_\Omega(r).
\end{align}
\end{proof}

The following properties follow almost immediately from \Cref{lemma:matheron} and the theory developed in \cref{sec:bv_on_manifolds}. They are consistent with the properties of the autocorrelation function with flat underlying geometry (see \cite{KnuShi:2023, BrKnMC:2023}).

\begin{proposition} \label{prop:autocorrelation_function}
    Let $\Omega \subset M$ be measurable and $r,s \in \IR$.
    Then the autocorrelation function $c_\Omega$ satisfies the following properties.
    \begin{enumerate}
        \item (Uniform Bounds) $0 \leq c_\Omega(r) \leq c_\Omega(0) = |\Omega|$.
        \item (Reflection invariance) $c_\Omega(r) = c_\Omega(-r)$.
        \item (Sum Estimate) $c_\Omega(r) + c_\Omega(s) \le c_\Omega(0) + c_\Omega(r-s)$.
        \item (Pre-Lipschitz estimate) $|c_\Omega(r) - c_\Omega(s)| \le c_\Omega(0) - c_\Omega(r-s)$.
        \item (Complement relation) $c_{\Omega^c}(r) = |M| - 2|\Omega| + c_\Omega(r)$.
    \end{enumerate}
\end{proposition}

\begin{proof}~
    Using Matheron's formula established by \Cref{lemma:matheron}, statements a)--d) follow immediately from the corresponding properties of the geodesic variation in \Cref{prop:geodesic_variation_properties}.
    The last assertion follows from the relation $\ind_{\Omega^c} = 1 -\ind_\Omega$ together with a direct computation.
\end{proof}

We are now in the position to state the characterisation of sets of finite perimeter in terms of the autocorrelation function.
As a reminder, we say that $\Omega \subset M$ is a set of finite perimeter, if $\ind_\Omega \in \mathrm{BV}(M)$. In that case we write $\Per(\Omega) \coloneqq V[\ind_\Omega]$.

\begin{theorem} \label{thm:autocorrelation_function_Lipschitz}
    Let $\Omega \subset M$ be measurable.
    Then the following statements are equivalent:
    \begin{enumerate}[itemsep=1mm]
        \item $\Omega$ is a set of finite perimeter.
        \item The autocorrelation function $c_\Omega$ is Lipschitz continuous.
    \end{enumerate}
    In that case, the right-sided derivative at $r=0$ is
    \begin{align}
        c'_\Omega(0)
        = -\norm{ c'_\Omega }_{L^\infty(\IR_{\ge0})}
        = - \tfrac12 k_n \Per(\Omega).
    \end{align}
\end{theorem}
\begin{proof}
    For any $r > 0$, \Cref{lemma:matheron} states that $c_\Omega(0) - c_\Omega(r) = \frac 12 G_{\ind_\Omega}(r)$. The theorem then follows from \Cref{thm:geodesic_variation_Lipschitz}.
\end{proof}

We remark that the constant factor $\frac{k_n}2$ is the exact same constant as in the case of flat spaces (see \cite{Galerne:2011, KnuShi:2023, BrKnMC:2023}). In other words: no geometric information is stored in this constant (other than dimension). Instead, the influence of the Riemannian metric on the autocorrelation function to first order is stored in the perimeter.

\section{Application to Pattern Formation}
\label{sec:application_pattern_formation}

In this section, we analyse the asymptotic behaviour of non-local isoperimetric energies on $S^n$ as the interaction length scale $\eps \to 0$, using the techniques developed above. More precisely, let $n \in \IN$, $n \geq 2$, and let $M = S^n$ be the $n$-dimensional unit sphere equipped with the round metric $g$. We denote by $d(x,y)$ the corresponding geodesic distance between $x,y \in M$, and by $\mathrm dx, \di y$ etc. the corresponding volume form. Let $0 < \theta < |M|$ and define the class of admissible sets
\begin{align}
    \AA \coloneqq \{ \Omega \subset M : \ind_\Omega \in BV(M), \ |\Omega| = \theta\}.
\end{align}
Let $\gamma, \eps > 0$. We define the non-local isoperimetric energy
\begin{align}
    E_{\gamma,\eps}(\Omega) \coloneqq \left\{ \begin{array}{ll}
        \displaystyle \Per(\Omega) - \frac{\gamma}{\eps} \int_{M \times M} K_\eps(x,y) \, \bigl| \ind_\Omega(x) - \ind_\Omega(y) \bigr| \di x \di y & \quad \text{if } \Omega \in \AA, \\[10pt]
        +\infty & \quad \text{otherwise,}
    \end{array} \right.
\end{align}
where $K_\eps \colon M \times M \longrightarrow \IR$ is the fundamental solution of the Helmholtz equation
    \begin{align}
        K_\eps(x,y) - \eps ^2 \Delta_x K_\eps(x,y) = \delta_y && \text{for all } x,y \in M,
    \end{align}
where $\Delta$ denotes the Laplace-Beltrami operator with respect to the metric $g$.
The parameter $\gamma$ controls the strength of the non-local interaction relative to the perimeter, while $\eps$ sets the length scale of the interaction.
Such energies arise in models of pattern formation where systems tend to minimise interfaces while maintaining interactions between separated regions \cite{BrKnMC:2023}.

We define the quantity $\gamma_\crit \geq 0$ via
\begin{align} \label{eq:defi_gamma_crit}
    \gamma_\crit \coloneqq \lim_{\eps \to 0} \gamma_\eps,  && \frac{1}{\gamma_\eps} \coloneqq k_n  \int_M \frac{d(x,y)}{\eps} \, K_\eps(x,y) \di x
\end{align}

for some $y \in M$. We show in \Cref{prop:convergence_gamma_eps} that this quantity is independent of $y \in M$ and $0 < \gamma_\crit < \infty$ is given explicitly. In this section, we prove the following $\Gamma$--convergence result regarding the non-local isoperimetric energy $E_{\gamma, \eps}$, where for two measurable sets $\Omega, \Omega' \subset M$, we denote the symmetric difference by $\Omega \bigtriangleup \Omega' \coloneqq (\Omega \setminus \Omega') \cup (\Omega' \setminus \Omega)$ and its volume by $|\Omega \bigtriangleup \Omega'|$.

\begin{theorem} \label{thm:gamma_convergence}
    Let $0 < \gamma < \gamma_{\crit}$. Then $E_{\gamma,\eps} \stackrel{\Gamma}{\longrightarrow} E_{\gamma,0}$ in the $L^1$ topology, where
    \begin{align}
        E_{\gamma,0}(\Omega) \coloneqq \left\{ \begin{array}{lll}
            \bigl( 1 - \frac{\gamma}{\gamma_{\crit}} \bigr) \Per(\Omega)  && \quad \text{if } \Omega \in \AA, \\[6pt]
            +\infty  && \quad \text{else.}
        \end{array}\right.
    \end{align}
    More precisely:
    \begin{itemize}
        \item \textbf{Liminf inequality:} For every measurable $\Omega \subset M$ and for every sequence of measurable sets $\Omega_\eps \subset M$ such that $|\Omega \bigtriangleup \Omega_\eps| \to 0$ we have
        \begin{align}
            E_{\gamma,0}(\Omega) \leq \liminf_{\eps \to 0} E_{\gamma,\eps}(\Omega_\eps).
        \end{align}
        \item \textbf{Limsup inequality:} For every measurable $\Omega \subset M$ there exists a sequence of measurable sets $\Omega_\eps \subset M$ such that $|\Omega \bigtriangleup \Omega_\eps| \to 0$ and
        \begin{align}
            E_{\gamma,0}(\Omega) \geq \limsup_{\eps \to 0} E_{\gamma,\eps}(\Omega_\eps).
        \end{align}
    \end{itemize}
\end{theorem}

The strategy to prove \Cref{thm:gamma_convergence} is as follows. In \Cref{lemma:reformulation_energy_to_autocorrelation1} and \Cref{lemma:reformulation_energy_to_autocorrelation2}, we reformulate the energy in terms of the autocorrelation function (similar to \cite[Lemma 4.2]{BrKnMC:2023}, see also \cite[Proposition 2.6]{KnuShi:2023}). Using the sharp bounds on the autocorrelation function, we find optimal lower bounds of the energy in terms of the perimeter functional. Compactness (see \Cref{cor:compactness}) as well as the liminf inequality follow from these sharp bounds. In addition, this reformulation quantifies the error from the limit functional to zeroth order (see also \cite[Lemma 3.1]{MuNoSi:2025}). For the limsup inequality, we again use the reformulation in terms of the autocorrelation function together with explicit computations regarding integrals of the Helmholtz kernel. We find the pointwise limit of the energy $E_{\gamma,\eps}$ as $\eps \to 0$ (see \Cref{prop:energy_pointwise_convergence}) and use the constant sequence as a recovery sequence to conclude the proof of \Cref{thm:gamma_convergence}.

We remark that the the case of the sphere with radius $R$ follows from the case of the unit sphere by rescaling.
To be more precise, consider the the energy
    \begin{align}
        E^{(R)}_{\gamma,\eps}(\Omega) \coloneqq \Per^{(R)}(\Omega) - \frac{\gamma}{\eps} \int_{M^{(R)} \times M^{(R)}} K_{\eps}^{(R)}(x,y) \, |\ind_\Omega(x) - \ind_\Omega(y)| \di x^{(R)} \di y^{(R)},
    \end{align}
where $M^{(R)} \coloneqq (S^n, R^2 g)$ is the $n$-dimensional round sphere of radius $R$, $\Per^{(R)}$ denotes the perimeter with respect to the metric $R^2 g$, $\mathrm dx^{(R)}, \di y^{(R)}$ denote the volume forms with respect to the metric $R^2 g$, and $K_\eps^{(R)} \colon M^{(R)} \times M^{(R)} \longrightarrow \IR$ is the solution of
    \begin{align}
        K_\eps^{(R)}(\filll,y) - \eps^2 \Delta^{(R)} K_\eps^{(R)}(\filll,y) = \delta_y,
    \end{align}
where $\Delta^{(R)}$ is the Laplace-Beltrami operator corresponding to the metric $R^2 g$. Simple calculations reveal that
    \begin{align}
        \Per^{(R)} = R^{n - 1} \Per,  && \mathrm dx^{(R)} = R^n \mathrm dx, && K_\eps^{(R)} = \frac{1}{R^n} K_{\eps/R},
    \end{align}
and thus
    \begin{align}
        E^{(R)}_{\gamma,\eps} = R^{n - 1} E_{\gamma, \frac{\eps}{R}} \stackrel{\Gamma}{\longrightarrow} R^{n - 1} \big(1 - \frac{\gamma}{\gamma_\crit} \big) \Per = \big(1 - \frac{\gamma}{\gamma_\crit} \big) \Per^{(R)}.
    \end{align}
In particular, the value $\gamma_\crit$ is independent of the radius of the sphere.

\subsection{Kernel on the Sphere} \label{sec:helmholtz_kernel}

In this section, we take a closer look at the Helmholtz kernel $K_\eps \colon M \times M \longrightarrow \IR$ defined as the unique solution to the equation
\begin{align} \label{eq:helmholtz_equation}
    K_\eps(\filll, y) - \eps^2 \Delta K_\eps(\filll,y) = \delta_y  && \text{in } M,
\end{align}
where the equation is to be understood in the distributional sense, and $\delta_y$ denotes the Dirac distribution with support $y \in M$. The operator $\Delta$ denotes the Laplace-Beltrami operator (with respect to the round metric). It follows from standard elliptic theory that there exists a unique solution to the equation \eqref{eq:helmholtz_equation} which is smooth outside the diagonal (see also \cite[Theorem 4.8]{CoDaDu:2018} where the solution is given explicitly).

\begin{proposition} \label{prop:helmholtz_kernel_properties}
    Let $\eps > 0$ and let $K_\eps$ be the Helmholtz kernel. Then:
    \begin{enumerate}
        \item $\displaystyle \int_M K_\eps(x,y) \di x = 1$ for all $y \in M$.
        \item $K_\eps$ only depends on the distance, i.e. $K_\eps(x,y) = K_\eps\bigl( d(x,y) \bigr)$ for all $x,y \in M$.
        \item $K_\eps \geq 0$.
        \item As $\eps \to 0$ it holds
            \begin{align}
              \int_M \frac{d(x,y)}{\eps} K_\eps(x,y) \di y + \int_M \frac{d(x,y)^2}{\eps^2} K_\eps(x,y) \di y = \O(1). 
            \end{align}
    \end{enumerate}
\end{proposition}
\begin{proof} 
    The first assertion immediately follows from testing \eqref{eq:helmholtz_equation} with the constant function.
    Assertion b) follows from the explicit representation of the Helmholtz kernel proved in \cite[Theorem 4.8]{CoDaDu:2018}.

    \medskip
    
    To show assertion c), we note that since the integral over $K_\eps(\filll, y)$ is positive, we know that the set $U_y \coloneqq \{x \in M : K_\eps(x,y) \geq 0\} \neq \emptyset$. Let $V_y \coloneqq \{x \in M : K_\eps(x,y) \leq 0\}$. If $V_y = \emptyset$ then there is nothing to show. If $V_y \neq \emptyset$, it is easy to see that $y \notin V_y$. Moreover $V_y$ is compact and
    \begin{align}
        \eps^2 \Delta_x K_\eps(x,y) = K_\eps(x,y) \leq 0  && \text{for all } x \in V_y,
    \end{align}
    so $K_\eps(\filll,y)$ is superharmonic in $V_y$. It follows from the strong maximum principle (see e.g. \cite[Theorem 7.1.7]{Peters:2016}) that $K_\eps(\filll,y) = 0$ in $V_y$, and thus $K_\eps \geq 0$.

    \medskip

    To see assertion d), we first note that it suffices to show the bounds for only one $y \in M$ since the kernel only depends on the distance. We test \eqref{eq:helmholtz_equation} with $\varphi(x) \coloneqq \frac{d^2(x,y)}{\eps^2}$ and find
    \begin{align}
        \int_M \frac{d(x,y)^2}{\eps^2} K_\eps(x,y) \di x  = \int_M \Delta_x (d(x,y)^2) \, K_\eps(x,y) \di x.
    \end{align}
    Using the identity $\Delta_x d(x,y)^2 = 2 + 2(n - 1) \, d(x,y) \, \cot(d(x,y))$ (see e.g. \cite[example~3.23, eq.~(3.84)]{Grigoryan2012}), we further compute
    \begin{align}
        \int_M \Delta_x (d(x,y)^2) \, K_\eps(x,y) \di x  & = 2 + 2 (n - 1) \int_{M_+} d(x,y) \, \cot(d(x,y)) \, K_\eps(x,y) \di x    \\
            & \quad + 2(n - 1) \int_{M_-} d(x,y) \, \cot(d(x,y)) \, K_\eps(x,y) \di x,
    \end{align}
    where $M_+ = B_{\frac \pi 2}(y)$ and $M_- = M \setminus M_+$. For the integration in $M_+$, we note that $r \longmapsto r \cot(r)$ is bounded by $1$ for $0 < r < \frac \pi 2$, and thus 
    \begin{align}
        \int_{M_+} d(x,y) \, \cot(d(x,y)) \, K_\eps(x,y) \di x \leq \int_{M_+} K_\eps(x,y) \di x \leq 1.
    \end{align}
    For the integral over $M_-$, we note that since $K_\eps \xrightarrow{\eps \to 0} 0$ pointwise, for every $\eps_0 > 0$ there exists a constant $\m(\eps_0) > 0$ such that $K_\eps \leq \m$ in $M_-$ for every $0 < \eps < \eps_0$. Thus we find after employing polar coordinates (see \eqref{eq:polar-coordinates-2})
    \begin{align}
        \int_{M_-} d(x,y) \, \cot(d(x,y)) \, K_\eps(x,y) \di x   & \leq \m \int_{M_-} d(x,y) \, \cot(d(x,y)) \di x  \\[6pt]
            & = \m \, \sigma_{n - 1} \int_{\frac{\pi}{2}}^\pi r \cot(r) \sin(r)^{n - 1} \di r < \infty,
    \end{align}
    where we recall $\sigma_{n - 1}$ is the volume of the unit sphere (see \eqref{eq:sphere_volume_formula}). Altogether, the second moment is uniformly bounded provided that $0 < \eps < \eps_0$. Having obtained the uniform bound on the second moment, we find that
    \begin{align}
        \int_M \frac{d(x,y)}{\eps} K_\eps(x,y) \di x  & = \int_{B_{\eps}(y)} \frac{d(x,y)}{\eps} K_\eps(x,y) \di x + \int_{M \setminus B_{\eps}(y)} \frac{d(x,y)}{\eps} K_\eps(x,y) \di x   \\[6pt]
            & \leq \int_{B_{\eps}(y)} K_\eps(x,y) \di x + \int_{M \setminus B_{\eps}(y)} \frac{d(x,y)^2}{\eps^2} K_\eps(x,y) \di x,
    \end{align}
    and thus is uniformly bounded provided $0 < \eps < \eps_0$. 
\end{proof}

We now show that the first moments are converging and thus, that the critical value $\gamma_\crit$ defined in \eqref{eq:defi_gamma_crit} is well defined.

\begin{proposition} \label{prop:convergence_gamma_eps}
    Let $\gamma_\eps$ as in \eqref{eq:defi_gamma_crit}. Then $\gamma_\eps$ is independent of $y$ and a convergent sequence as $\eps \to 0$ with positive limit $\gamma_\crit = 1$.
\end{proposition}
\begin{proof} 
    It suffices to show that the first moment of $K_\eps$ converges as $\eps\to 0$. Let $0 < \eps < \eps_0 \coloneqq \frac{2}{n - 1}.$
    From \Cref{prop:helmholtz_kernel_properties} d) we know that there exists a constant $\m_1 > 0$ that only depends on $n$ such that
    \begin{align}
        \int_M \frac{d(x,y)}{\eps} K_\eps(x,y) \di x \leq \m_1 && \text{for all } 0 < \eps < \eps_0.
    \end{align}
    Since $K_\eps$ only depends on distance, the first moments are independent of $y \in M$ and thus $\gamma_\eps$ is independent of $y$.
    If we set $\beta = \frac1\eps$ in \cite[Theorem 4.8]{CoDaDu:2018},
    this theorem yields an explicit representation formula of the Helmholtz kernel $\eps^2 K_\eps$.
    Thus, for $0 < \eps < \eps_0$ our kernel takes the form
    \begin{align}
        K_\eps(r) = \frac1{\eps^2} \frac{\Gamma(\mu+\nu+1) \, \Gamma(\mu-\nu)}{2 \, (2\pi)^{\mu+1} \sin(r)^\mu} \, \mathrm P^{-\mu}_\nu \bigl( -\cos(r) \bigr)   && \text{for all } 0 < r < \pi,
    \end{align}
    where $\mathrm P^{-\mu}_\nu \colon (-1,1) \longrightarrow \IR$ is the Ferrers conical function of the first kind, and 
    \begin{align}
        \mu = \frac{n}2 - 1, && \nu = -\frac12 + \i \tau,   && \tau^2 \coloneqq \frac1{\eps^2} - \frac{(n-1)^2}4,
    \end{align}
    Note that in the limit $\eps \to 0$ we can expand $\tau = \frac1\eps + \O(1) \to \infty$ or equivalently $\eps\tau = 1 + \O(\eps) \to 1$, and in particular $\O(\tfrac 1\tau) = \O(\eps)$.
    The Ferrers function is defined in terms of the hypergeometric function $_2F_1$ (see \cite[eq.~(2.35)]{CoDaDu:2018}),
    but its exact form is not relevant in this article.
    However, its asymptotic expansion for $\tau\to\infty$ stated in \cite[Theorem 2.17, (2.64)]{CoDaDu:2018} provides us with the following expansion of our Helmholtz kernel: For every $0 < \rho < \pi$ there exists $\tau_0 > 0$, such that for all $\tau > \tau_0$
    \begin{align} \label{eq:asymptotic_ferrers_functions}
        K_\eps(r)
        & = \frac{a_\tau \tau^\mu}{(2\pi)^{\mu+1}\eps^2} \frac1{\sin(r)^\mu} \sqrt{\frac{r}{\sin(r)}} \bessel_\mu(\tau r) \, \big(1 + \O(\tfrac 1\tau)\big)
        && \text{for all } 0 < r < \rho
    \\
        & \text{with }
        a_\tau \coloneqq \Gamma(\mu+\nu+1) \, \Gamma(\mu-\nu) \frac{e^{\pi \tau}}{2\pi \tau^{n-2}}\ ,
    \end{align}
    where $\bessel_\mu \colon \IR_{> 0} \longrightarrow \IR$ denotes the modified Bessel function of the second kind. Its exact form is also not relevant in this article.
    We note that in our situation $\mu - \nu = (\mu + \nu + 1)^*$, where $z^*$ denotes the complex conjugate for $z$. Thus\footnote{Empty products are set to $1$.}
    \begin{align}
        \Gamma(\mu+\nu+1) \, \Gamma(\mu-\nu) = \bigl| \Gamma(\tfrac{n-1}2 + \i\tau) \bigr|^2
        & = \left\{ \begin{array}{ll}
            \displaystyle \frac{\pi}{\cosh(\pi\tau)} \prod_{j = 1}^{\frac{n}2 - 1}  \big((j - \tfrac12)^2 + \tau^2\big)
            & \quad n \in 2\IN,
        \\[16pt]
            \displaystyle \frac{\pi \tau}{\sinh(\pi\tau)} \prod_{j = 1}^{\frac{n-1}2 - 1} \bigl( j^2 + \tau^2 \bigr)
            & \quad n \in 2\IN + 1,
        \end{array}\right.
    \end{align}
    and it follows $a_\tau = 1 + \O(\tfrac 1\tau) = 1 + \O(\eps) \to 1$ in the limit $\eps \to 0$.
    
    To compute the limit of $\gamma_\eps$, let $\delta > 0$. Using \Cref{prop:helmholtz_kernel_properties} we find $0 < \rho < \pi$ such that
    \begin{align}
        \int_{M \setminus B_\rho(y)} \frac{d(x,y)}{\eps} K_\eps(x,y) \di x < \delta    && \text{for all } 0 < \eps < \eps_0,
    \end{align}
    where $B_\rho(y)$ denotes the geodesic ball around $y$ with radius $\rho$.
    So to compute the desired first moment, we can focus on the domain around the reference point $y \in M$.
    Using polar normal coordinates as in \eqref{eq:polar-coordinates-2} and the asymptotic expansion \eqref{eq:asymptotic_ferrers_functions}, we compute for $\tau > \tau_0$
    \begin{align}
        \int_{B_\rho(y)} \frac{d(x,y)}{\eps} K_\eps(x,y) \di x
        & = \frac{\sigma_{n-1}}{\eps} \int_0^\rho r \, K_\eps(r) \sin(r)^{n-1}\di r
    \\
        & = \frac{\sigma_{n-1}}{(2\pi)^{\mu+1}} \frac{a_\tau \tau^\mu}{\eps^3} \int_0^\rho r \sin(r)^\frac{n}2 \sqrt{\frac{r}{\sin(r)}} \bessel_\mu(\tau r) \, \big(1 + \O(\tfrac1\tau)\big) \di r
    \\
        & = \frac {\sigma_{n-1}}{(2\pi)^{\mu+1}} \frac{a_\tau}{(\eps\tau)^3} \int_0^\rho \bigl( \tau r \bigr)^{\frac{n}2+1} \biggl( \frac{\sin(r)}{r} \biggr)^\frac{n-1}2 \bessel_\mu(\tau r) \, \tau \di r \, \big(1 + \O(\tfrac1\tau)\big)
    \\
        & = \frac {\sigma_{n-1}}{(2\pi)^{\mu+1}} \int_0^{\rho \tau} r^{\frac{n}2+1} \biggl( \frac{\sin(r/\tau)}{r/\tau} \biggr)^\frac{n-1}2 \bessel_\mu(r) \di r
        \, \big(1 + \O(\tfrac1\tau)\big)
    \end{align}
    Since the function $r \longmapsto r^{\frac n2 + 1} \bessel_\mu(r)$ is integrable over $\IR_{\ge0}$, it follows from the dominated convergence theorem that
    \begin{align}
        \int_{B_\rho(y)} \frac{d(x,y)}{\eps} K_\eps(x,y) \di x
        & \xrightarrow{\eps \to 0} \frac{\sigma_{n - 1}}{(2\pi)^{\mu+1}} \ubr{\int_0^\infty r^{\frac n2 + 1} \bessel_\mu(r) \di r}{=}{\sqrt \pi \, 2^{\frac{n}2 - 1} \Gamma(\frac{n+1}2)}
        = \frac{1}{k_n},
    \end{align}
    where in the last line we used \Cref{prop:sphere_integration}.
    The integral of the Bessel function is recorded in \cite[p.676]{GraRyz:2015}. Taking the limits $\rho \to \pi$ and then $\delta \to 0$ we obtain the desired convergence. Combining the results of the calculation together with \eqref{eq:defi_gamma_crit} and \Cref{prop:sphere_integration} we obtain
    \begin{align}
        \gamma_\crit    & = \lim_{\eps \to 0}\bigg( k_n \int_M \frac{d(x,y)}{\eps} K_\eps(x,y) \di x \bigg)^{-1} = 1,
    \end{align}
    which was the claim.
\end{proof}

Since the Helmholtz kernel only depends on the distance (see \Cref{prop:helmholtz_kernel_properties} b)), we are able to integrate it radially to define the integrated kernel $\Phi_\eps \colon (0,\pi) \longrightarrow \IR$ via
\begin{align} \label{eq:defi_Phi_eps}
   \Phi_\eps(r) \coloneqq \frac {\sigma_{n - 1}}\eps \int_r^{\pi} K_\eps(r) \, \sin(r)^{n - 1} \di r && \text{for all } r \in (0,\pi).
\end{align}
The following properties of $\Phi_\eps$ follow from the properties of the Helmholtz kernel.

\begin{proposition} \label{prop:properties_Phi_eps}
    Let $\Phi_\eps \colon (0,\pi) \longrightarrow \IR$ given as in \eqref{eq:defi_Phi_eps}. Then
    \begin{enumerate}
        \item $\Phi_\eps \geq 0$.
        \item $\Phi_\eps(0) = \frac 1 \eps$ and $\Phi_\eps(\pi) = 0$.
        \item $\|\Phi_\eps\|_{L^1(I)} = \frac {1}{k_n \gamma_\eps}$.
    \end{enumerate}
\end{proposition}
\begin{proof} 
    Since $K_\eps \geq 0$, so is $\Phi_\eps$. The second assertions follows from the integrability of $K_\eps$. To show the third assertion, we integrate by parts to find
    \begin{align}
        \int_0^\pi \Phi_\eps(r) \di r  & =  \Big[r \, \Phi_\eps(r) \Big]_{r = 0}^{r = \pi} -\int_0^\pi r \Phi_\eps'(r) \di r  \\[6pt]
            & = \frac{\sigma_{n-1}}{\eps} \int_0^\pi r \sin(r)^{n-1} K_\eps(r) \di r   \\[6pt]
            & = \int_M \frac{d(x,y)}{\eps} K_\eps(x,y) \di x = \frac1{k_n \gamma_\eps},
    \end{align}
    where we used polar normal coordinates in the last line.
\end{proof}

\subsection{Asymptotic Analysis}

In this section, we present the proof of \Cref{thm:gamma_convergence}. We first reformulate the non-local term in $E_{\gamma,\eps}$ in terms of the autocorrelation function. Since the Helmholtz kernel only depends on the distance, the autocorrelation function appears naturally in the non-local term of the energy.

\begin{lemma}  \label{lemma:reformulation_energy_to_autocorrelation1}
    Let $\Omega \subset M$ be measurable. Then, for every $\eps > 0$,
    \begin{align}
        \int_{M \times M} K_\eps(x,y)  \, \bigl| \ind_\Omega(x) - \ind_\Omega(y) \bigr| \di x \di y = 2 \sigma_{n - 1} \int_0^\pi K_\eps(r) \, \sin(r)^{n - 1} \bigl( c_\Omega(0) - c_\Omega(r) \bigr) \di r .
    \end{align}
\end{lemma}
\begin{proof}
    Since the Helmholtz kernel only depends on the distance, we introduce polar normal coordinates as in~\eqref{eq:polar-coordinates-2} and apply Tonelli's theorem:
    \begin{align}
        \int_{M \times M} & K_\eps(x,y) \, |\ind_\Omega(x) - \ind_\Omega(y)| \di x \di y
    \\[6pt]
        & = \sigma_{n-1} \int_0^\pi K_\eps(r) \, \sin(r)^{n - 1} \bigg( \int_M \oint_y |\ind_\Omega(\exp_y(rw)) - \ind_\Omega(y)| \di w \di y \bigg) \di r
    \\[6pt]
        & = 2 \sigma_{n-1} \int_0^\pi K_\eps(r) \, \sin(r)^{n - 1} \bigl( c_\Omega(0) - c_\Omega(r) \bigr) \di r,
    \end{align}
    where the last step is Matheron's formula presented in \Cref{lemma:matheron}.
\end{proof}

The reformulation in \Cref{lemma:reformulation_energy_to_autocorrelation1} is essential in proving the limsup inequality in \Cref{thm:gamma_convergence}. We show that the energy functional is pointwise converging to the limit functional, which enables us to use the constant sequence as a recovery sequence.

\begin{proposition} \label{prop:energy_pointwise_convergence}
    Let $\gamma > 0$ and $\Omega \in \AA$. Then
    \begin{align}
        E_{\gamma,\eps}(\Omega) \xrightarrow{\eps \to 0} E_{\gamma,0}(\Omega).
    \end{align}
\end{proposition}
\begin{proof}
    It suffices to prove convergence of the non-local term. Let $0 < \eps < \eps_0 \coloneqq \frac{2}{n - 1}$. Let $\delta > 0$ and $y \in M$. Using \Cref{prop:helmholtz_kernel_properties} we find $\rho > 0$ such that
    \begin{align} \label{eq:lemma:moment_by_delta}
        \int_{M \setminus B_\rho(y)} \frac{d(x,y)}{\eps} K_\eps(x,y) \di x < \delta && \text{for all } 0 < \eps < \eps_0.
    \end{align}
    Using \Cref{lemma:reformulation_energy_to_autocorrelation1}, we find after switching to polar normal coordinates (see \eqref{eq:polar-coordinates-2})
    \begin{align}
        \frac{1}{\eps} \int_{M \times M} K_\eps(x,y) \, \bigl| \ind_\Omega(x) - \ind_\Omega(y) \bigr| \di x \di y 
        &= \frac{2 \sigma_{n - 1}}{\eps} \int_0^\rho K_\eps(r) \sin(r)^{n - 1} \, \bigl( c_\Omega(0) - c_\Omega(r) \bigr) \di r  \\[6pt]
        &\phantom{={}}\ + 2 \sigma_{n - 1}\int_\rho^\pi \frac{r}{\eps} K_\eps(r) \sin(r)^{n - 1} \, \frac{c_\Omega(0) - c_\Omega(r)}{r} \di r.
    \end{align}
    Using the Lipschitz continuity of the autocorrelation function with corresponding Lipschitz constant $\frac{k_n}{2} \Per(\Omega)$ (see \Cref{thm:autocorrelation_function_Lipschitz}), as well as \eqref{eq:lemma:moment_by_delta}, the second integral can be uniformly estimated as follows:
    \begin{align}
        \bigg| 2 \sigma_{n - 1}\int_\rho^\pi \frac{r}{\eps} K_\eps(r) \, \sin(r)^{n - 1} \, \frac{c_\Omega(0) - c_\Omega(r)}{r} \di r \, \bigg| \leq \delta \, k_n \Per(\Omega)   && \text{for all } 0 < \eps < \eps_0.
    \end{align}
    For the other integrand, with the notation set in the proof of \Cref{prop:convergence_gamma_eps}, we rewrite 
    \begin{align}
        \frac{2 \sigma_{n - 1}}{\eps} & \int_0^\rho K_\eps(r) \sin(r)^{n - 1} \, \bigl( c_\Omega(0) - c_\Omega(r) \bigr) \di r
    \\[6pt]
        & = \frac{2 \sigma_{n-1}}{(2 \pi)^{\mu+1}} 
        \int_0^{\rho \tau} r^{\frac n2 + 1} \ \bigg(\frac{\sin(r /\tau)}{r/\tau} \bigg)^\frac{n-1}2 \bessel_\mu(r) \, \frac{c_\Omega(0) - c_\Omega(r/\tau)}{r/\tau} 
        \di r \, \big(1 + \O(\tfrac1\tau)\big)
    \\[6pt]
        & \xrightarrow{\eps \to 0} \frac{2 \sigma_{n-1}}{(2 \pi)^{\mu+1}} 
        \int_0^\infty r^{\frac n2 + 1} \bessel_\mu(r) \, \bigl( -c'_\Omega(0) \bigr) \di r = \frac{1}{\gamma_\crit} \Per(\Omega),
    \end{align}
    where in the last line we used that $c_\Omega$ is Lipschitz continuous as well as the formula for $c'_\Omega(0)$ (see \Cref{thm:autocorrelation_function_Lipschitz}), together with the dominated convergence theorem.
    Since $\delta > 0$ was arbitrary, the claim follows.
\end{proof}

Using \Cref{lemma:reformulation_energy_to_autocorrelation1}, we find yet another reformulation of $E_{\gamma,\eps}$ in terms of the autocorrelation function using integration by parts. To do so, we recall the integrated kernel $\Phi_\eps \colon (0,\pi) \longrightarrow \IR$ in \eqref{eq:defi_Phi_eps}, which is given by
    \begin{align}
        \Phi_\eps(r) \coloneqq \frac {\sigma_{n - 1}}\eps \int_r^{\pi} K_\eps(r) \, \sin(r)^{n - 1} \di r && \text{for all } r \in (0,\pi).
    \end{align}
The following representation quantifies the error of $E_{\gamma,\eps}$ from the limit functional $E_{\gamma,0}$ in terms of the integrated kernel $\Phi_\eps$.

\begin{lemma} \label{lemma:reformulation_energy_to_autocorrelation2}
    Let $\gamma > 0$ and $\Omega \subset M$ a set of finite perimeter. Then
    \begin{align}
        E_{\gamma,\eps}(\Omega) = \Big(1 - \frac{\gamma}{\gamma_\eps} \Big) \Per(\Omega) + 2\gamma \int_0^\pi \Phi_\eps(r) \, \bigl( c'_\Omega(r) - c'_\Omega(0) \bigr) \di r .
    \end{align}
\end{lemma}
\begin{proof}
    Since $\Omega$ is a set of finite perimeter, we find that $c_\Omega$ is Lipschitz (see \Cref{thm:autocorrelation_function_Lipschitz}). Using \Cref{lemma:reformulation_energy_to_autocorrelation1} and performing an integration by parts yields
    \begin{align}
        \begin{multlined}
        \frac{\sigma_{n-1}}{\eps}\int_0^\pi K_\eps(r) \, \sin(r)^{n - 1} \, \bigl( c_\Omega(0) - c_\Omega(r) \bigr) \di r \\[6pt]
            = -\bigg[\Phi_\eps(r) \, \bigl( c_\Omega(0) - c_\Omega(r) \bigr) \bigg]_0^\pi - \int_0^\pi \Phi_\eps(r) \, c'_\Omega(r) \di r 
            = - \int_0^\pi \Phi_\eps(r) \, c'_\Omega(r) \di r,
        \end{multlined}
    \end{align}
    where in the last step \Cref{prop:properties_Phi_eps} is used. Using again \Cref{thm:autocorrelation_function_Lipschitz} we further obtain
    \begin{align}
        \int_0^\pi \Phi_\eps(r) \, c'_\Omega(r) \di r
        = \int_0^\pi \Phi_\eps(r) \, \bigl( c'_\Omega(r) - c'_\Omega(0) \bigr) \di r - \frac{k_n}{2} \|\Phi_\eps\|_{L^1(I)} \Per(\Omega).
    \end{align}
    In total we get
    \begin{align}
        E_{\gamma,\eps}(\Omega) = \Per(\Omega) + 2 \gamma \int_0^\pi \Phi_\eps(r) \bigl( c'_\Omega(r) - c'_\Omega(0) \bigr) \di r - \frac{\gamma}{\gamma_\eps} \Per(\Omega),
    \end{align}
    which was the claim.
\end{proof}

With the reformulation in \Cref{lemma:reformulation_energy_to_autocorrelation2} at hand, we formulate two simple corollaries. The first one is a pointwise lower bound of the energy in terms of the perimeter functional. It is the crucial step in bounding the energy optimally from below, which is a result of the fine bounds obtained for the autocorrelation function. This pointwise lower bound is used to obtain the liminf inequality in \Cref{thm:gamma_convergence}.

\begin{corollary} \label{cor:pointwise_lower_bound}
    Let $\gamma, \eps > 0$ and $\Omega \in \AA$. Then 
        \begin{align}
            \Big(1 - \frac{\gamma}{\gamma_\eps} \Big) \Per(\Omega) \leq E_{\gamma,\eps}(\Omega).
        \end{align}
\end{corollary}
    \begin{proof}
        Let $\Omega \in \AA$. Using \Cref{thm:autocorrelation_function_Lipschitz} we find that $c'(r) \geq c'(0)$ for almost all $r > 0$. Using \Cref{prop:properties_Phi_eps}a) and \Cref{lemma:reformulation_energy_to_autocorrelation2} we find
            \begin{align}
                E_{\gamma,\eps}(\Omega) = \Big(1 - \frac{\gamma}{\gamma_\eps} \Big) \Per(\Omega) + 2\gamma \int_0^\pi \Phi_\eps(r) \, \bigl( c_\Omega'(r) - c_\Omega'(0) \bigr) \di r \geq \Big(1 - \frac{\gamma}{\gamma_\eps} \Big) \Per(\Omega)
            \end{align}
        as claimed.
    \end{proof}

In addition to the pointwise lower bound, another corollary of the reformulation in \Cref{lemma:reformulation_energy_to_autocorrelation2} is the $L^1$-compactness of the energy functional in the subcritical regime.

\begin{corollary} \label{cor:compactness}
    Let $0 < \gamma < \gamma_\crit$. Let $\Omega_\eps \subset M$ measurable such that $\limsup_{\eps \to 0} E_{\gamma,\eps}(\Omega_\eps) < \infty$. Then there exists $\Omega \in \AA$ and a subsequence (not relabeled) such that $|\Omega \bigtriangleup \Omega_\eps| \xrightarrow{\eps \to 0} 0$.
\end{corollary}
\begin{proof}
    From the energy bound we can assume, without loss of generality, that $\Omega_\eps \in \AA$ for all $\eps > 0$.
    Since $0 < \gamma < \gamma_\crit$, it follows from \Cref{cor:pointwise_lower_bound} that, for $\eps$ sufficiently small,
        \begin{align}
            0 \leq \Big(1 - \frac{\gamma}{\gamma_\eps} \Big) \Per(\Omega_\eps) \leq E_{\gamma,\eps}(\Omega_\eps).
        \end{align}
    Thus $\Per(\Omega_\eps)$ is uniformly bounded as $\eps \to 0$. Since $|\Omega_\eps| = \theta$ for all $\eps > 0$, the claim follows from the compact embedding of $BV(M)$ into $L^1(M)$ (see \Cref{prop:bv_compact}). The limit is again a set of finite perimeter since, after going to a subsequence, the pointwise limit of characteristic functions is again a characteristic function.
\end{proof}

We are now in the position to prove the main result \Cref{thm:gamma_convergence}.

\medskip

\begin{proof}[Proof of \Cref{thm:gamma_convergence}:]
    We show the liminf and limsup inequality seperately.
    For the limsup inequality, we use the constant sequence $\Omega_\eps = \Omega$ as a recovery sequence.
    Without loss of generality, we may assume that $\Omega \in \AA$ since otherwise the limsup inequality is trivial.
    The claim then follows from \Cref{prop:energy_pointwise_convergence}.

    \medskip

    To show the liminf inequality, let $\Omega_\eps, \Omega \subset M$ such that $|\Omega_\eps \bigtriangleup \Omega| \xrightarrow{\eps \to 0} 0$. If $\liminf_{\eps \to 0} E_{\gamma, \eps}(\Omega_\eps) = +\infty$, then the liminf inequality is trivially satisfied. Otherwise we obtain from \Cref{cor:compactness} that $\Omega \in \AA$. We find using the lower-semicontinuity of the perimeter (see \Cref{prop:bv_lower_semicontinuous}) and the super-multiplicativity of the limit inferior
    \begin{align}
        E_{\gamma,0}(\Omega)    & \leq \biggl( \lim_{\eps \to 0} \Bigl( 1 - \frac{\gamma}{\gamma_\eps} \Bigr) \biggr) \biggl( \liminf_{\eps \to 0} \Per(\Omega_\eps) \biggr) \\[6pt]
            & \leq \liminf_{\eps \to 0} \biggl( \Bigl(1 - \frac{\gamma}{\gamma_\eps} \Bigr) \Per(\Omega_\eps) \biggr) \\[6pt]
            & \leq \liminf_{\eps \to 0} E_{\gamma,\eps}(\Omega_\eps),
    \end{align}
    where in the last line we used \Cref{cor:pointwise_lower_bound}.
\end{proof}

\bibliographystyle{alpha}
\bibliography{bibliography}

\end{document}